\documentclass[12pt]{article}
\usepackage{amsmath, amsthm, amsfonts}
\usepackage[utf8]{inputenc}
\usepackage{tikz}
\usepackage{verbatim}
\usepackage{graphicx}
\usepackage[mathcal]{euscript}
\usepackage{float}
\usepackage{cite}
\usepackage[normalem]{ulem}
\usepackage{colortbl}
\usepackage{mathtools}
\usepackage{environ}
 \usetikzlibrary{plotmarks}
\allowdisplaybreaks
\usetikzlibrary{decorations.pathreplacing}
\usepackage[toc,page]{appendix}
\usepackage{comment}
\usepackage[normalem]{ulem}
\usepackage{geometry}

\newtheorem{theorem}{Theorem}
\newtheorem{lemma}[theorem]{Lemma}
\newtheorem{corollary}[theorem]{Corollary}
\newtheorem{proposition}[theorem]{Proposition}

\theoremstyle{definition}
\newtheorem{definition}[theorem]{Definition}

\newcommand{\mf}[1]{\mathfrak{#1}}

\theoremstyle{remark}
\newtheorem{remark}[theorem]{Remark}

\numberwithin{equation}{section}

\def\mf{\mathfrak}

\title{The $2$-factor polynomial detects even perfect matchings}

\author{Scott Baldridge\\
\small Department of Mathematics\\[-0.8ex]
\small Louisiana State University\\[-0.8ex]
\small Baton Rouge, LA, U.S.A.\\
\small\tt sbaldrid@math.lsu.edu\\
\and
Adam M. Lowrance\thanks{The second author was supported in part by NSF grant DMS-1811344.}\\
\small Department of Mathematics and Statistics\\[-0.8ex]
\small Vassar College\\[-0.8ex]
\small Poughkeepsie, NY, U.S.A.\\
\small\tt adlowrance@vassar.edu\\
\and
Ben McCarty\\
\small Department of Mathematical Sciences\\[-0.8ex]
\small University of Memphis\\[-0.8ex]
\small Memphis, TN, U.S.A.\\
\small\tt ben.mccarty@memphis.edu}

\begin{document}

\maketitle

\begin{abstract}
In this paper, we prove that the $2$-factor polynomial, an invariant of a planar trivalent graph with a perfect matching, counts the number of $2$-factors that contain the perfect matching as a subgraph. Consequently, we show that the polynomial detects even perfect matchings.  
\end{abstract}

\section{Introduction}

In \cite{BaldCohomology}, the first author developed a cohomology theory of a planar trivalent graph equipped with a perfect matching, and observed that the graded Euler characteristic of that cohomology is a polynomial invariant called the $2$\emph{-factor polynomial}.  Given  a planar trivalent graph $G$ with a perfect matching $M$, a perfect matching drawing $\Gamma$ for the pair $(G,M)$ is a topological embedding of the graph in the $2$-sphere that marks the perfect matching edges.  We define the $2$-factor bracket of the perfect matching drawing, written $\langle \Gamma \rangle_2$ as the Laurent polynomial in the variable $z$ calculated recursively on immersed perfect matching drawings by the relations in Figure \ref{figure:2factordef}.

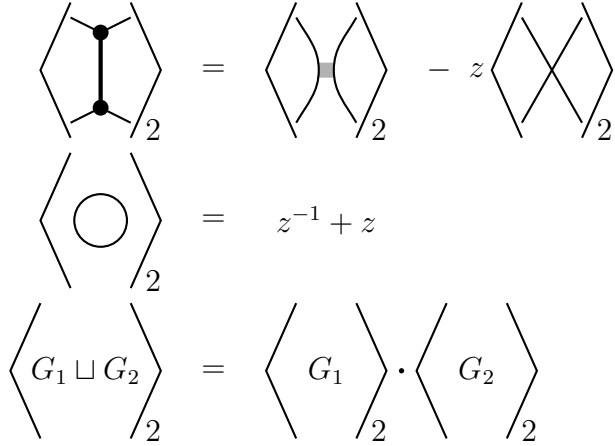
\begin{figure}[h]
$$\begin{tikzpicture}[thick, scale = 1]

\draw[ultra thick] (0,0) -- (0,1);
\fill (0,0) circle (3pt);
\fill (0,1) circle (3pt);
\draw (-.4,-.2) -- (0,0) -- (.4,-.2);
\draw (-.4,1.2) -- (0,1) -- (.4,1.2);
\draw (-.4,-.4) -- (-.8,.5) -- (-.4,1.4);
\draw (.4,-.4) -- (.8,.5) -- (.4,1.4);
\draw (.7,-.3) node{$2$};

\draw (1.5, .5) node {$=$};

\begin{scope}[xshift = 3cm]

	\fill[white!70!black] (.12,.6) to 
	(-.12,.6) to [out = -60, in = 90] 
	(-.1,.5) to [out = -90, in = -120] 
	(-.12,.4) to 
	(.12,.4) to [out = 120, in = -90]
	(.1,.5) to [out = 90, in = -120]
	(.12,.6);
	\draw (-.4,-.4) -- (-.8,.5) -- (-.4,1.4);
	\draw (.4,-.4) -- (.8,.5) -- (.4,1.4);
	\draw (.7,-.3) node{$2$};
	\draw (-.4,-.2) to [out = 60, in = -90] (-.1,.5) to [out = 90, in = -60] (-.4,1.2);
	\draw (.4,-.2) to [out = 120, in = -90] (.1,.5) to [out = 90, in = -120] (.4,1.2);

\end{scope}


\draw (4.5, .5) node{$-$};
\draw (5,.5) node{$z$};

\begin{scope}[xshift = 6 cm]
	\draw (-.4,-.4) -- (-.8,.5) -- (-.4,1.4);
	\draw (.4,-.4) -- (.8,.5) -- (.4,1.4);
	\draw (.7,-.3) node{$2$};
	\draw (-.4,-.2) -- (.4,1.2);
	\draw (-.4,1.2) -- (.4,-.2);

\end{scope}

\begin{scope}[yshift = -2cm]

\draw (0,.5) circle (.35cm);
\draw (-.4,-.4) -- (-.8,.5) -- (-.4,1.4);
\draw (.4,-.4) -- (.8,.5) -- (.4,1.4);
\draw (.7,-.3) node{$2$};

\draw (1.5, .5) node {$=$};

\draw (3,.5) node{$z^{-1} + z$};

\end{scope}

\begin{scope}[yshift = -4cm]

\draw (-.2,.5) node {$G_1\sqcup G_2$};
\draw (-.8,-.4) -- (-1.2,.5) -- (-.8,1.4);
\draw (.4,-.4) -- (.8,.5) -- (.4,1.4);
\draw (.7,-.3) node{$2$};

	\draw (1.5, .5) node {$=$};

\begin{scope}[xshift = 3cm]
	
	\draw (0,.5) node{$G_1$};
	\draw (-.4,-.4) -- (-.8,.5) -- (-.4,1.4);
	\draw (.4,-.4) -- (.8,.5) -- (.4,1.4);
	\draw (.7,-.3) node{$2$};

\end{scope}


\fill (4, .5) circle (1pt);

\begin{scope}[xshift = 5 cm]
	\draw (-.4,-.4) -- (-.8,.5) -- (-.4,1.4);
	\draw (.4,-.4) -- (.8,.5) -- (.4,1.4);
	\draw (.7,-.3) node{$2$};
	\draw (0,.5) node{$G_2$};

\end{scope}

\end{scope}

\end{tikzpicture}$$
\caption{ The recursive rules for computing the $2$-factor polynomial.}
\label{figure:2factordef}
\end{figure}

The bold edge in Figure \ref{figure:2factordef} represents a perfect matching edge (the only edges to which we apply the bracket), and the circle represents any immersed circle with no vertices.  

In  Section 3 of \cite{BaldCohomology}, it was shown that the $2$-factor polynomial is independent of the choice of perfect matching drawing, and thus is an invariant of planar trivalent graphs equipped with a perfect matching, i.e., the pair $(G,M)$.  Thus we can define the $2$-factor polynomial of the pair $(G,M)$ to be the $2$-factor bracket of the perfect matching drawing $\Gamma$.
$$\langle G : M \rangle_2(z) = \langle \Gamma \rangle_2(z).$$

The bracket relations in Figure \ref{figure:2factordef} give a recursive method for computing the $2$-factor polynomial. For example, the $2$-factor polynomial of the $\theta$ graph with perfect matching is:

\begin{eqnarray*}
\Big\langle \raisebox{-0.43\height}{\includegraphics[scale=0.30]{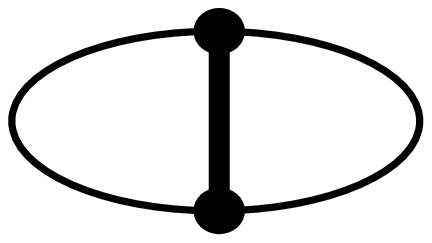}} \Big{\rangle}_{\!\!2} &=& \Big \langle \raisebox{-0.43\height}{\includegraphics[scale=0.20]{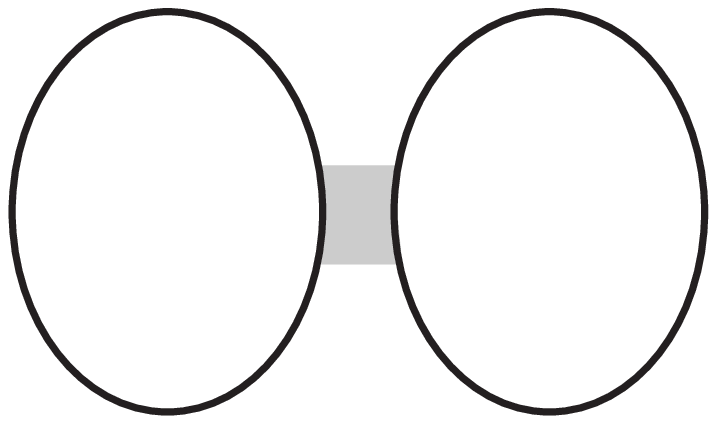}} \Big\rangle_{\!\!2} -z \Big\langle \raisebox{-0.43\height}{\includegraphics[scale=0.20]{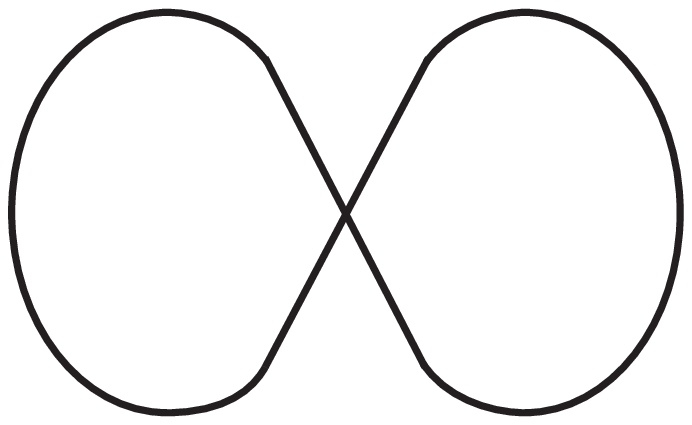}} \Big\rangle_{\!\!2}\\
&=& (z^{-1}+z)^2 - z(z^{-1}+z)\\
&=& z^{-2}+1.\\
\end{eqnarray*}

The polynomial depends on both the graph and the perfect matching.  To see this dependence, we compute the $2$-factor polynomial of the $3$\emph{-prism}, or \emph{circular ladder}, denoted $P_3$, for each of the perfect matchings in Figure~\ref{fig:3prism}.  

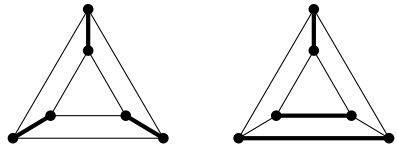
\begin{figure}[h]
$$\begin{tikzpicture}[scale = 1]

\begin{scope}[]
\draw[ultra thick] (-0.5,-0.5) -- (-1,-0.8);
\draw[ultra thick] (0.5,-0.5) -- (1,-0.8);
\draw[ultra thick] (0,0.866-0.5) -- (0,0.866+0.05);
\draw (-1,-0.8) -- (0,0.866+0.05);
\draw (1,-0.8) -- (0,0.866+0.05);
\draw (-1,-0.8) -- (1,-0.8);
\fill (-1,-0.8) circle (2pt);
\fill (1,-0.8) circle (2pt);
\fill (0,0.866+0.05) circle (2pt);
\fill (-0.5,-0.5) circle (2pt);
\fill (0.5,-0.5) circle (2pt);
\fill (0,0.866-0.5) circle (2pt);
\draw (-0.5,-0.5) -- (0.5,-0.5) -- (0,0.866-0.5) -- cycle;

\end{scope}

\begin{scope}[xshift = 3 cm]
\draw (-0.5,-0.5) -- (-1,-0.8);
\draw (0.5,-0.5) -- (1,-0.8);
\draw[ultra thick] (0,0.866-0.5) -- (0,0.866+0.05);
\draw (-1,-0.8) -- (0,0.866+0.05);
\draw (1,-0.8) -- (0,0.866+0.05);
\draw[ultra thick] (-1,-0.8) -- (1,-0.8);
\fill (-1,-0.8) circle (2pt);
\fill (1,-0.8) circle (2pt);
\fill (0,0.866+0.05) circle (2pt);
\fill (-0.5,-0.5) circle (2pt);
\fill (0.5,-0.5) circle (2pt);
\fill (0,0.866-0.5) circle (2pt);
\draw (-0.5,-0.5) -- (0.5,-0.5) -- (0,0.866-0.5) -- cycle;
\draw[ultra thick] (-0.5,-0.5) -- (0.5,-0.5);

\end{scope}

\end{tikzpicture}$$
\caption{A $3$-prism with perfect matchings $L$ and $C$.}
\label{fig:3prism}
\end{figure}

The calculation of these examples also serve as a way to introduce the \emph{cube of resolutions}, which we will use throughout this paper (cf. \cite{BaldCohomology}).  Briefly, to construct a cube of resolutions for $(P_3,L)$, first resolve each perfect matching of the graph into one of the two resolutions shown in the first equation of Figure~\ref{figure:2factordef} to get a state.  Next, order the states into columns from left-to-right by the number of ``cross-resolutions'' $\bigtimes$ in each state.  Finally, draw an arrow from a state in one column to a state in the adjacent column to the right if the two states are related by replacing one ``open-resolution'' $)($ with one cross-resolution $\bigtimes$.  The cube of resolutions for $(P_3,L)$ is given in Figure~\ref{fig:3PrismCube}.

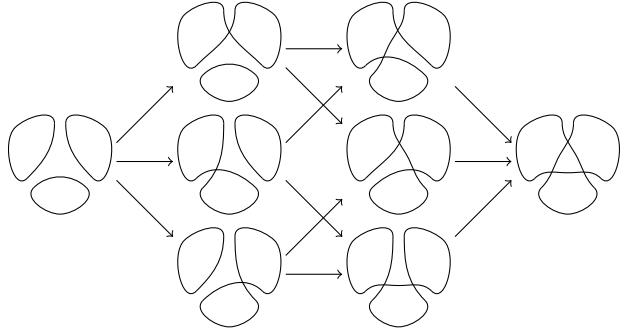
\begin{figure}[H]
$$\begin{tikzpicture}[scale = 0.5]

\begin{scope}
\draw (0.15,0.866-0.05) to [out = -90, in = 135] (0.8 + 0.08,-0.8 + 0.15);
\draw (-0.15,0.866-0.05) to [out = -90, in = 45] (-0.8 - 0.08,-0.8 + 0.15);
\draw (-0.8 + 0.095,-0.8 - 0.1) to [out = 45, in = 135] (0.8 - 0.095,-0.8 - 0.1);

\draw (0.15,0.866-0.05) to [out = 90, in = 135] (1.2,0.7) to [out = -45, in = -45] (0.8 + 0.08,-0.8 + 0.15);
\draw (-0.15,0.866-0.05) to [out = 90, in = 45] (-1.20, 0.7) to [out = 225, in = 225] (-0.8 - 0.08,-0.8 + 0.15);
\draw (0.8 - 0.095,-0.8 - 0.1) to [out = -45, in = 0] (0, -1.6) to [out = 180, in = 225]  (-0.8 + 0.095,-0.8 - 0.1);
\end{scope}

\begin{scope}[xshift = 1 cm, yshift = -0.2 cm]
\draw[->] (0.5,0.5) --(2,2);
\draw[->] (0.5,0) --(2,0);
\draw[->] (0.5,-0.5) --(2,-2);
\end{scope}

\begin{scope}[xshift = 5.5 cm, yshift = -0.2 cm]
\draw[->] (0.5,0.5) --(2,2);
\draw[->] (0.5,-0.5) --(2,-2);
\draw[->] (0.5,3) --(2,3);
\draw[->] (0.5,-3) --(2,-3);
\draw[->] (0.5,-2.5) --(2,-1);
\draw[->] (0.5,2.5) --(2,1);
\end{scope}

\begin{scope}[xshift = 8.5 cm, yshift = -0.2 cm]
\draw[->] (2,2) -- (3.5,0.5);
\draw[->] (2,0) -- (3.5,0);
\draw[->] (2,-2) -- (3.5,-0.5);
\end{scope}

\begin{scope}[xshift = 4.5 cm, yshift = 3 cm]
\draw (-0.15,0.866-0.05) to [out = -90, in = 135] (0.8 + 0.08,-0.8 + 0.15);
\draw (0.15,0.866-0.05) to [out = -90, in = 45] (-0.8 - 0.08,-0.8 + 0.15);
\draw (-0.8 + 0.095,-0.8 - 0.1) to [out = 45, in = 135] (0.8 - 0.095,-0.8 - 0.1);

\draw (0.15,0.866-0.05) to [out = 90, in = 135] (1.2,0.7) to [out = -45, in = -45] (0.8 + 0.08,-0.8 + 0.15);
\draw (-0.15,0.866-0.05) to [out = 90, in = 45] (-1.20, 0.7) to [out = 225, in = 225] (-0.8 - 0.08,-0.8 + 0.15);
\draw (0.8 - 0.095,-0.8 - 0.1) to [out = -45, in = 0] (0, -1.6) to [out = 180, in = 225]  (-0.8 + 0.095,-0.8 - 0.1);
\end{scope}

\begin{scope}[xshift = 4.5 cm, yshift = 0 cm]
\draw (0.15,0.866-0.05) to [out = -90, in = 135] (0.8 + 0.08,-0.8 + 0.15);
\draw (-0.15,0.866-0.05) to [out = -90, in = 45] (-0.8 + 0.095,-0.8 - 0.1);
\draw (-0.8 - 0.08,-0.8 + 0.15) to [out = 45, in = 135] (0.8 - 0.095,-0.8 - 0.1);
\draw (0.15,0.866-0.05) to [out = 90, in = 135] (1.2,0.7) to [out = -45, in = -45] (0.8 + 0.08,-0.8 + 0.15);
\draw (-0.15,0.866-0.05) to [out = 90, in = 45] (-1.20, 0.7) to [out = 225, in = 225] (-0.8 - 0.08,-0.8 + 0.15);
\draw (0.8 - 0.095,-0.8 - 0.1) to [out = -45, in = 0] (0, -1.6) to [out = 180, in = 225]  (-0.8 + 0.095,-0.8 - 0.1);
\end{scope}

\begin{scope}[xshift = 4.5 cm, yshift = -3 cm]
\draw (0.15,0.866-0.05) to [out = -90, in = 135] (0.8 - 0.095,-0.8 - 0.1);
\draw (-0.15,0.866-0.05) to [out = -90, in = 45] (-0.8 - 0.08,-0.8 + 0.15);
\draw (-0.8 + 0.095,-0.8 - 0.1) to [out = 45, in = 135] (0.8 + 0.08,-0.8 + 0.15);
\draw (0.15,0.866-0.05) to [out = 90, in = 135] (1.2,0.7) to [out = -45, in = -45] (0.8 + 0.08,-0.8 + 0.15);
\draw (-0.15,0.866-0.05) to [out = 90, in = 45] (-1.20, 0.7) to [out = 225, in = 225] (-0.8 - 0.08,-0.8 + 0.15);
\draw (0.8 - 0.095,-0.8 - 0.1) to [out = -45, in = 0] (0, -1.6) to [out = 180, in = 225]  (-0.8 + 0.095,-0.8 - 0.1);
\end{scope}

\begin{scope}[xshift = 4.5 cm]
\begin{scope}[xshift = 4.5 cm, yshift = 3 cm]
\draw (-0.15,0.866-0.05) to [out = -90, in = 135] (0.8 + 0.08,-0.8 + 0.15);
\draw (0.15,0.866-0.05) to [out = -90,  in = 60] (-0.2, 0) to [out = 240, in = 45] (-0.8 + 0.095,-0.8 - 0.1);
\draw (-0.8 - 0.08,-0.8 + 0.15) to [out = 45, in = 135] (0.8 - 0.095,-0.8 - 0.1);
\draw (0.15,0.866-0.05) to [out = 90, in = 135] (1.2,0.7) to [out = -45, in = -45] (0.8 + 0.08,-0.8 + 0.15);
\draw (-0.15,0.866-0.05) to [out = 90, in = 45] (-1.20, 0.7) to [out = 225, in = 225] (-0.8 - 0.08,-0.8 + 0.15);
\draw (0.8 - 0.095,-0.8 - 0.1) to [out = -45, in = 0] (0, -1.6) to [out = 180, in = 225]  (-0.8 + 0.095,-0.8 - 0.1);
\end{scope}

\begin{scope}[xshift = 4.5 cm, yshift = 0 cm]
\draw (-0.15,0.866-0.05) to [out = -90, in = 120] (0.2, 0) to [out = -60, in = 135] (0.8 - 0.095,-0.8 - 0.1);
\draw (0.15,0.866-0.05) to [out = -90, in = 45] (-0.8 - 0.08,-0.8 + 0.15);
\draw (-0.8 + 0.095,-0.8 - 0.1) to [out = 45, in = 135] (0.8 + 0.08,-0.8 + 0.15);
\draw (0.15,0.866-0.05) to [out = 90, in = 135] (1.2,0.7) to [out = -45, in = -45] (0.8 + 0.08,-0.8 + 0.15);
\draw (-0.15,0.866-0.05) to [out = 90, in = 45] (-1.20, 0.7) to [out = 225, in = 225] (-0.8 - 0.08,-0.8 + 0.15);
\draw (0.8 - 0.095,-0.8 - 0.1) to [out = -45, in = 0] (0, -1.6) to [out = 180, in = 225]  (-0.8 + 0.095,-0.8 - 0.1);
\end{scope}

\begin{scope}[xshift = 4.5 cm, yshift = -3 cm]
\draw (0.15,0.866-0.05) to [out = -90, in = 135] (0.8 - 0.095,-0.8 - 0.1);
\draw (-0.15,0.866-0.05) to [out = -90, in = 45] (-0.8 + 0.095,-0.8 - 0.1);
\draw (-0.8 - 0.08,-0.8 + 0.15) to [out = 45, in = 180] (0, -0.5) to [out = 0, in = 135] (0.8 + 0.08,-0.8 + 0.15);
\draw (0.15,0.866-0.05) to [out = 90, in = 135] (1.2,0.7) to [out = -45, in = -45] (0.8 + 0.08,-0.8 + 0.15);
\draw (-0.15,0.866-0.05) to [out = 90, in = 45] (-1.20, 0.7) to [out = 225, in = 225] (-0.8 - 0.08,-0.8 + 0.15);
\draw (0.8 - 0.095,-0.8 - 0.1) to [out = -45, in = 0] (0, -1.6) to [out = 180, in = 225]  (-0.8 + 0.095,-0.8 - 0.1);
\end{scope}

\end{scope}

\begin{scope}[xshift = 13.5 cm]
\begin{scope}
\draw (-0.15,0.866-0.05) to [out = -90, in = 120] (0.2, 0) to [out = -60, in = 135] (0.8 - 0.095,-0.8 - 0.1);
\draw (0.15,0.866-0.05) to [out = -90,  in = 60] (-0.2, 0) to [out = 240, in = 45] (-0.8 + 0.095,-0.8 - 0.1);
\draw (-0.8 - 0.08,-0.8 + 0.15) to [out = 45, in = 180] (0, -0.5) to [out = 0, in = 135] (0.8 + 0.08,-0.8 + 0.15);
\draw (0.15,0.866-0.05) to [out = 90, in = 135] (1.2,0.7) to [out = -45, in = -45] (0.8 + 0.08,-0.8 + 0.15);
\draw (-0.15,0.866-0.05) to [out = 90, in = 45] (-1.20, 0.7) to [out = 225, in = 225] (-0.8 - 0.08,-0.8 + 0.15);
\draw (0.8 - 0.095,-0.8 - 0.1) to [out = -45, in = 0] (0, -1.6) to [out = 180, in = 225]  (-0.8 + 0.095,-0.8 - 0.1);
\end{scope}

\end{scope}

\end{tikzpicture}$$
\caption{The cube of resolutions for $(P_3,L)$.}
\label{fig:3PrismCube}
\end{figure}

To calculate the polynomial, note that each state contributes a term of the form $(-z)^k(z+z^{-1})^\ell$ where $k$ is the number of cross-resolutions $\bigtimes$ and $\ell$ is the number of immersed circles in the state.   From Figure \ref{fig:3PrismCube}, we can calculate the polynomial of $(P_3,L)$ directly to see that 
\begin{equation*}
\begin{aligned}
\langle P_3 : L \rangle_2(z) &  = (z+z^{-1})^3 - 3z (z+z^{-1})^2 + 3z^2 (z+z^{-1}) - z^3 (z+z^{-1}) \\
& = z^{-3} - z^2+z^3-z^4. \\
\end{aligned}
\end{equation*}

\begin{figure}[H]
$$\begin{tikzpicture}[scale = 0.5]

\begin{scope}
\draw (0.15,0.866-0.05) to [out = -90, in = 0] (0.3,-0.3) to [out = 180, in = 0] (0,-0.3);
\draw (-0.15,0.866-0.05) to [out = -90, in = 180] (-0.3, -0.3) to [out = 0, in = 180] (0,-0.3);
\draw (-0.5, -0.7) to [out = 90, in = 180] (0,-0.5) to [out = 0, in = 90] (0.5, -0.7);
\draw (-0.5, -0.7) to [out = 270, in = 180] (0,-1) to [out = 0, in = 270] (0.5, -0.7);

\draw (0.15,0.866-0.05) to [out = 90, in = 135] (0.8,0.7) to [out = -45, in = 0] (0.8,-1.2);
\draw (-0.15,0.866-0.05) to [out = 90, in = 45] (-0.8,0.7) to [out = 225, in = 180] (-0.8,-1.2);
\draw (-0.8,-1.2) -- (0.8,-1.2);
\end{scope}

\begin{scope}[xshift = 1 cm, yshift = -0.2 cm]
\draw[->] (0.5,0.5) --(2,2);
\draw[->] (0.5,0) --(2,0);
\draw[->] (0.5,-0.5) --(2,-2);
\end{scope}

\begin{scope}[xshift = 5.5 cm, yshift = -0.2 cm]
\draw[->] (0.5,0.5) --(2,2);
\draw[->] (0.5,-0.5) --(2,-2);
\draw[->] (0.5,3) --(2,3);
\draw[->] (0.5,-3) --(2,-3);
\draw[->] (0.5,-2.5) --(2,-1);
\draw[->] (0.5,2.5) --(2,1);
\end{scope}

\begin{scope}[xshift = 8.5 cm, yshift = -0.2 cm]
\draw[->] (2,2) -- (3.5,0.5);
\draw[->] (2,0) -- (3.5,0);
\draw[->] (2,-2) -- (3.5,-0.5);
\end{scope}

\begin{scope}[xshift = 4.5 cm, yshift = 3 cm]
\draw (-0.15,0.866-0.05) to [out = -90, in = 0] (0.3,-0.3) to [out = 180, in = 0] (0,-0.3);
\draw  (0.15,0.866-0.05) to [out = -90, in = 180] (-0.3, -0.3) to [out = 0, in = 180] (0,-0.3);
\draw (-0.5, -0.7) to [out = 90, in = 180] (0,-0.5) to [out = 0, in = 90] (0.5, -0.7);
\draw (-0.5, -0.7) to [out = 270, in = 180] (0,-1) to [out = 0, in = 270] (0.5, -0.7);

\draw (0.15,0.866-0.05) to [out = 90, in = 135] (0.8,0.7) to [out = -45, in = 0] (0.8,-1.2);
\draw (-0.15,0.866-0.05) to [out = 90, in = 45] (-0.8,0.7) to [out = 225, in = 180] (-0.8,-1.2);
\draw (-0.8,-1.2) -- (0.8,-1.2);
\end{scope}

\begin{scope}[xshift = 4.5 cm, yshift = 0 cm]
\draw (0.15,0.866-0.05) to [out = -90, in = 0] (0.3,-0.3) to [out = 180, in = 90] (-0.5,-0.7);
\draw (-0.15,0.866-0.05) to [out = -90, in = 180] (-0.3, -0.3) to [out = 0, in = 90] (0.5,-0.7);
\draw (-0.5, -0.7) to [out = 270, in = 180] (0,-1) to [out = 0, in = 270] (0.5, -0.7);

\draw (0.15,0.866-0.05) to [out = 90, in = 135] (0.8,0.7) to [out = -45, in = 0] (0.8,-1.2);
\draw (-0.15,0.866-0.05) to [out = 90, in = 45] (-0.8,0.7) to [out = 225, in = 180] (-0.8,-1.2);
\draw (-0.8,-1.2) -- (0.8,-1.2);
\end{scope}

\begin{scope}[xshift = 4.5 cm, yshift = -3 cm]
\draw (0.15,0.866-0.05) to [out = -90, in = 0] (0.3,-0.3) to [out = 180, in = 0] (0,-0.3);
\draw (-0.15,0.866-0.05) to [out = -90, in = 180] (-0.3, -0.3) to [out = 0, in = 180] (0,-0.3);
\draw (-0.5, -0.7) to [out = 90, in = 180] (0,-0.5) to [out = 0, in = 90] (0.5, -0.7);
\draw (-0.5, -0.7) to [out = 270, in = 180] (0.8,-1.2);

\draw (0.15,0.866-0.05) to [out = 90, in = 135] (0.8,0.7) to [out = -45, in = 0] (0.8,-1.2);
\draw (-0.15,0.866-0.05) to [out = 90, in = 45] (-0.8,0.7) to [out = 225, in = 180] (-0.8,-1.2);
\draw (-0.8,-1.2) to [out = 0, in = 270] (0.5, -0.7) ;
\end{scope}

\begin{scope}[xshift = 4.5 cm]
\begin{scope}[xshift = 4.5 cm, yshift = 3 cm]
\draw (-0.15,0.866-0.05) to [out = -90, in = 0] (0.3,-0.3) to [out = 180, in = 90] (-0.5,-0.7);
\draw  (0.15,0.866-0.05) to [out = -90, in = 180] (-0.3, -0.3) to [out = 0, in = 90] (0.5,-0.7);
\draw (-0.5, -0.7) to [out = 270, in = 180] (0,-1) to [out = 0, in = 270] (0.5, -0.7);

\draw (0.15,0.866-0.05) to [out = 90, in = 135] (0.8,0.7) to [out = -45, in = 0] (0.8,-1.2);
\draw (-0.15,0.866-0.05) to [out = 90, in = 45] (-0.8,0.7) to [out = 225, in = 180] (-0.8,-1.2);
\draw (-0.8,-1.2) -- (0.8,-1.2);
\end{scope}

\begin{scope}[xshift = 4.5 cm, yshift = 0 cm]
\draw (-0.15,0.866-0.05) to [out = -90, in = 0] (0.3,-0.3) to [out = 180, in = 0] (0,-0.3);
\draw  (0.15,0.866-0.05) to [out = -90, in = 180] (-0.3, -0.3) to [out = 0, in = 180] (0,-0.3);
\draw (-0.5, -0.7) to [out = 90, in = 180] (0,-0.5) to [out = 0, in = 90] (0.5, -0.7);
\draw (-0.5, -0.7) to [out = 270, in = 180] (0.8, -1.2);

\draw (0.15,0.866-0.05) to [out = 90, in = 135] (0.8,0.7) to [out = -45, in = 0] (0.8,-1.2);
\draw (-0.15,0.866-0.05) to [out = 90, in = 45] (-0.8,0.7) to [out = 225, in = 180] (-0.8,-1.2);
\draw (-0.8,-1.2) to [out = 0, in = 270] (0.5, -0.7);
\end{scope}

\begin{scope}[xshift = 4.5 cm, yshift = -3 cm]
\draw (0.15,0.866-0.05) to [out = -90, in = 0] (0.3,-0.3) to [out = 180, in = 0] (0,-0.3);
\draw (-0.15,0.866-0.05) to [out = -90, in = 180] (-0.3, -0.3) to [out = 0, in = 90] (0.5,-0.7);
\draw (-0.5, -0.7) to [out = 90, in = 180] (0.3, -0.3);
\draw (-0.5, -0.7) to [out = 270, in = 180] (0.8,-1.2);

\draw (0.15,0.866-0.05) to [out = 90, in = 135] (0.8,0.7) to [out = -45, in = 0] (0.8,-1.2);
\draw (-0.15,0.866-0.05) to [out = 90, in = 45] (-0.8,0.7) to [out = 225, in = 180] (-0.8,-1.2);
\draw (-0.8,-1.2) to [out = 0, in = 270] (0.5, -0.7) ;

\end{scope}

\end{scope}

\begin{scope}[xshift = 13.5 cm]
\begin{scope}
\draw (-0.15,0.866-0.05) to [out = -90, in = 0] (0.3,-0.3) to [out = 180, in = 0] (0,-0.3);
\draw (0.15,0.866-0.05) to [out = -90, in = 180] (-0.3, -0.3) to [out = 0, in = 90] (0.5,-0.7);
\draw (-0.5, -0.7) to [out = 90, in = 180] (0.3, -0.3);
\draw (-0.5, -0.7) to [out = 270, in = 180] (0.8,-1.2);

\draw (0.15,0.866-0.05) to [out = 90, in = 135] (0.8,0.7) to [out = -45, in = 0] (0.8,-1.2);
\draw (-0.15,0.866-0.05) to [out = 90, in = 45] (-0.8,0.7) to [out = 225, in = 180] (-0.8,-1.2);
\draw (-0.8,-1.2) to [out = 0, in = 270] (0.5, -0.7) ;
\end{scope}

\end{scope}

\end{tikzpicture}$$
\caption{The cube of resolutions for $(P_3,C)$.}
\label{fig:3PrismCube2}
\end{figure}
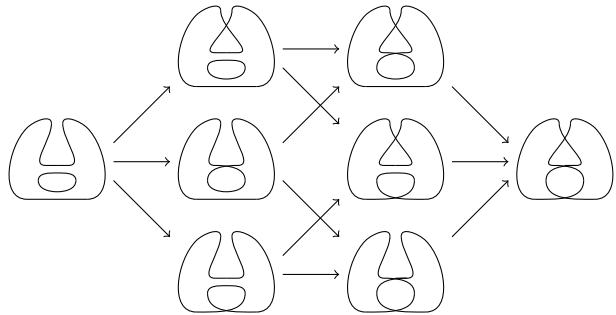

On the other hand, Figure \ref{fig:3PrismCube2} shows the cube of resolutions for $(P_3,C)$, and hence, we can calculate the polynomial directly to see that 
$$\langle P_3 : C \rangle_2(z) = z^{-2} - z^{-1} + 1 + z^3.$$
\noindent The polynomials for the two perfect matchings are different, and therefore the perfect matchings are different.  The polynomials depend upon the perfect matching. It is the goal of this paper to partially characterize that dependence.

The $2$-factor polynomial incorporates some of the most interesting results to  come out of knot theory in the past thirty years together with  counting techniques important to graph theory.  By evaluating the polynomials at 1, we see that ${\langle P_3 : L \rangle_2(1)}=0$ and $\langle P_3 : C \rangle_2(1)=2$.  The main theorem of this paper says that these numbers are the count of $2$-factors that contain  $L$ and $C$ respectively.  To describe the main theorem, we need to introduce some notions related to factorability.  

Recall that a $k$-factor of $G$ is a spanning $k$-regular subgraph of $G$.  For example, the vertex set $V(G)$ is a $0$-factor, and a perfect matching (with vertices) is a $1$-factor.  A $2$-factor is a disjoint set of a cycles that span the vertex set.  When $G$ is also regular, it is natural to count the number of $k$-factors ``between'' an $\ell$-factor $M$ of $G$ and $G$, i.e., the index of $k$-factors that contain $M$ and are contained in $G$.  

\begin{definition}
Given an $n$-regular graph $G$ and an $\ell$-factor $M$ of $G$, then a $k$-factor $K$ \emph{factors through} $M$ if $M$ is a subgraph of $K$. Denote the set of all $k$-factors of $G$ that factor through $M$ by $[G:M]$.  For a specific $k$, denote the set of all $k$-factors of $G$ that factor through $M$ by $[G:M]_k$ and denote the number of elements of $[G:M]_k$ by the index $|G:M|_k$.
\end{definition}

Our main theorem in this paper proves Conjecture 1.2 of \cite{BaldCohomology} and says that the $2$-factor polynomial counts the number of $2$-factors that factor through the perfect matching:

\begin{theorem}
\label{thm:MainTheorem}
Let $G$ be a planar trivalent graph with perfect matching $M$. Evaluating the $2$-factor polynomial of $(G,M)$ at $1$ counts the number of elements of $[G:M]_2$, i.e.,
$$\langle G:M\rangle_2(1) = |G:M|_2.$$
\end{theorem}

Looking back at the circular ladder examples, one can check by hand that $|P_3 : L|_2 = 0$, while $|P_3 : C|_2 = 2$, just as the theorem says.  Moreover, notice that in this example, $P_3 \setminus L$ has two $3$-cycles, while $P_3\setminus C$ contains a single $4$-cycle.  (Here and throughout we think of a $1$-factor as either a subgraph or a set of perfect matching edges, depending on context.) 

The graph $G\setminus M$ is always a set of disjoint cycles $\{C_1, C_2, \ldots, C_k\}$, and  the number of edges in each $C_i$ provides important information about $|G:M|_2$.  Consider one of these cycles, $C_i$, and note that every cycle of a $2$-factor that factors through $M$ must enter $C_i$ through a perfect matching edge of $M$ and then leave $C_i$ through a different perfect matching edge of $M$. Thus, any $2$-factor that factors through $M$ must span an even number of vertices of $C_i$.  If any cycle $C_i$ in $G\setminus M$ has an odd number of edges, i.e., if $C_i$ is an {\em odd cycle}, then $|G:M|_2 = 0$. 

On the other hand, if $C_i$ has an even number of edges for all $i$, then every $2$-factor that factors through $M$ consists of $M$ and a perfect matching for each $C_i$.  This is because every other edge of $C_i$ must be included in the $2$-factor.   Since each $C_i$ has two perfect matchings, there are $2^k$ choices.  Thus we obtain the following:

\begin{proposition}
\label{prop:count}
Let $G$ be a connected planar trivalent graph with perfect matching $M$ such that $G\setminus M$ is a set of $k$ disjoint cycles. If any cycle in this set is odd, then $|G:M|_2=0$. If all cycles are even, then $|G:M|_2= 2^k$.
\end{proposition}

Theorem \ref{thm:MainTheorem} and Propostion \ref{prop:count} taken together imply that the $2$-factor polynomial is $0$ whenever the graph $G\setminus M$ contains a cycle with an odd number of edges, and it is a power of two precisely when all cycles are even. Thus we immediately obtain:

\begin{corollary}
\label{odd}
Let $G$ be a planar trivalent graph with perfect matching
$M$. If any of the cycles in $G(V, E \setminus M)$ have an odd number of edges, then the $2$-factor polynomial
for $(G, M)$ satisfies
$${\langle {G:M}\rangle}_2(1)=0.$$
In particular, if $\langle {G:M} \rangle_2 (1) > 0$ for the pair $(G,M)$, then $M$ is an even perfect matching of $G$.
\end{corollary}

The existence of an even perfect matching for a planar trivalent graph implies that there is a Tait coloring for the graph, i.e., it is $3$-edge colorable (cf. \cite{KauffmanStateCalc}).  In fact, Theorem~\ref{thm:MainTheorem}  implies more. Pick three colors $\{\mf{i},\mf{j},\mf{k}\}$ to label edges. By Theorem~\ref{thm:MainTheorem}, $\langle {G:M} \rangle_2(1)$ counts the number of $2$-factors that factor through $M$.  For each such $2$-factor, label the perfect matching edges of the $2$-factor by $\mf{i}$ and the remaining edges of the $2$-factor by $\mf{j}$.  Label the edges in the complement of the $2$-factor by $\mf{k}$.  Thus, $\langle {G:M} \rangle_2(1)$ counts the number of Tait colorings of $G$ where the edges of $M$ are colored by $\mf{i}$. Thus, summing up over all perfect matchings counts the total number of Tait colorings of $G$.  This motivates the  definition:

\begin{definition} Let $G$ be a connected planar trivalent graph.  Define the {\em planar Tait polynomial of $G$} to be
$$T_G(z) = \sum_{M\in[G:V(G)]_1} \langle G:M \rangle_2(z).$$
\end{definition}

\noindent (Recall that $[G:V(G)]_1$ is the set of all $1$-factors of $G$.) The polynomial $T_G$ is an invariant of $G$ (cf.  \cite{BaldCohomology}).  In particular, one can prove the following about $T_G$:

\begin{theorem}
If $G$ is a connected planar trivalent graph, then $T_G(1)$ counts the number of Tait colorings of $G$, i.e.,
$$T_G(1) = \#|\text{Tait colorings of }G|.$$
\end{theorem}

For example, the theta graph $\theta$ has three perfect matchings, one for each edge of the graph.  The $2$-factor polynomial is the same for each perfect matching, thus
$$T_\theta(z)= 3z^{-2}+3.$$
 There are six Tait colorings of the $\theta$ graph, and the above computation of $T_\theta(z)$ implies $T_{\theta}(1)=6$.

Another example that we can calculate is $P_3$.  There is one perfect matching corresponding to the ladder $L$ and three perfect matchings corresponding to $C$ and two $120^\circ$ rotations of $C$. Therefore, 
$$T_{P_3}(z)=(z^{-3}-z^2+z^3-z^4)+3(z^{-2}-z^{-1}+1+z^3),$$
and $T_{P_3}(1)=6$, which is the number of Tait colorings of $P_3$. Note that the perfect matching $L$ does not contribute any colorings.

A planar trivalent graph $G$ has a Tait coloring if and only if its faces are $4$-colorable (cf. \cite{KauffmanStateCalc}, \cite{Heawood}, \cite{Kempe}). Hence, the four color theorem \cite{AppelHaken, AppelHaken2, RSST} is equivalent to the following theorem. For other examples of results equivalent to the four color theorem, see \cite{BN2} and \cite{Kauffman3}.

\begin{theorem}
\label{theorem:2FactorPoly}
If $G$ is a connected, bridgeless, planar trivalent graph, then $T_G(1)>0$.
\end{theorem}

The planar Tait polynomial is a sum of $2$-factor polynomials. It is the authors' hope that the $2$-factor polynomial or its categorification from \cite{BaldCohomology} may shed light on the number of Tait colorings of a connected planar trivalent graph.

\section{Preliminaries}

The proof of Theorem \ref{thm:MainTheorem} rests upon the use of $IH$-moves to transform a given bridgeless, planar trivalent graph with a perfect matching $(G,M)$ into one where we know Theorem \ref{thm:MainTheorem} holds.   

\begin{definition}
\label{def:IHMove}
An $IH$-move is a local replacement of the configuration shown on the left in Figure \ref{figure:IH} with the one shown on the right.
\end{definition}

\begin{figure}[h]
$$\begin{tikzpicture}[scale = 0.8]

\draw[dashed] (0,0.5) circle (0.8 cm);

\draw[ultra thick] (0,0) -- (0,1);
\fill (0,0) circle (3pt);
\fill (0,1) circle (3pt);
\draw (-.4,-.2) -- (0,0) -- (.4,-.2);
\draw (-.4,1.2) -- (0,1) -- (.4,1.2);

\begin{scope}[xshift = 1 cm]
\draw[->] (.5,.6) -- (1,.6);
\draw[<-] (.5,.4) -- (1,.4);
\end{scope}

\begin{scope}[xshift = 2cm]

\draw[dashed] (1.8,0.5) circle (0.8 cm);

\draw (1.1,.9) -- (1.3,.5) -- (1.1,.1);
\fill (1.3,.5) circle (3pt);
\fill (2.3,.5) circle (3pt);
\draw [ultra thick] (1.3,.5) -- (2.3,.5);
\draw (2.5,.9) -- (2.3,.5) -- (2.5,.1);
\end{scope}

\end{tikzpicture}$$
\caption{An IH-move.}
\label{figure:IH}
\end{figure}
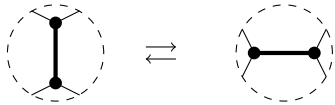

Note that we only perform an $IH$-move on a perfect matching edge.

\begin{definition}
\label{def:IHEquivalence}
Given two graphs, $G$ and $G'$, and perfect matchings for each $M$ and $M'$ respectively, we say that $(G,M)$ is $IH$-equivalent to $(G',M')$ if there is a sequence of $IH$-moves taking one to the other.
\end{definition}

Before we proving the main theorem, we need to establish a series of lemmas about $IH$-moves.

\begin{lemma}
\label{lem:IHLemma}
The $2$-factor polynomial and the number of $2$-factors of $G$ that factor through $M$ satisfy the following \textit{IH-relations}:

\begin{alignat}{7}
\tikz [thick, scale = .8, baseline=(current bounding box.center)]{\draw[ultra thick] (0,0) -- (0,1);
\fill (0,0) circle (3pt);
\fill (0,1) circle (3pt);
\draw (-.4,-.2) -- (0,0) -- (.4,-.2);
\draw (-.4,1.2) -- (0,1) -- (.4,1.2);
\draw (-.4,-.4) -- (-.8,.5) -- (-.4,1.4);
\draw (.4,-.4) -- (.8,.5) -- (.4,1.4);
\draw (.7,-.3) node{$2$};}
&  - & ~
\tikz[thick, scale = .8, baseline=(current bounding box.center)]{\draw (-0.5,.9) -- (-0.3,.5) -- (-0.5,.1);
\fill (-0.3,.5) circle (3pt);
\fill (0.7,.5) circle (3pt);
\draw [ultra thick] (-0.2,.5) -- (0.8,.5);
\draw (0.9,.9) -- (0.7,.5) -- (0.9,.1);
\draw (-.4,-.4) -- (-.8,.5) -- (-.4,1.4);
\draw (.8,-.4) -- (1.2,.5) -- (.8,1.4);
\draw (1.1,-.3) node{$2$};}
& = & ~
\tikz[thick, scale = .8, baseline=(current bounding box.center)]{
\fill[white!70!black] (.12,.6) to 
	(-.12,.6) to [out = -60, in = 90] 
	(-.1,.5) to [out = -90, in = -120] 
	(-.12,.4) to 
	(.12,.4) to [out = 120, in = -90]
	(.1,.5) to [out = 90, in = -120]
	(.12,.6);
	\draw (-.4,-.4) -- (-.8,.5) -- (-.4,1.4);
	\draw (.4,-.4) -- (.8,.5) -- (.4,1.4);
	\draw (.7,-.3) node{$2$};
	\draw (-.4,-.2) to [out = 60, in = -90] (-.1,.5) to [out = 90, in = -60] (-.4,1.2);
	\draw (.4,-.2) to [out = 120, in = -90] (.1,.5) to [out = 90, in = -120] (.4,1.2);}
& - & ~ 
\tikz[thick, scale = .8, baseline=(current bounding box.center)]{\draw (-.6,-.4) -- (-1,.5) -- (-.6,1.4);	
	\fill[white!70!black] 
	(.1,.6) to (-.1,.6) to (-.1,.4) to 
	(.1,.4) to (.1,.6);
	\draw (-.7,0.9) to [out = -30, in = 180] (0,.6) to [out = 0, in = 210] (0.7,0.9);
	\draw (-.7,0.1) to [out = 30, in = 180] (0,0.4) to [out = 0, in = 150] (0.7,0.1);
	\draw (.6,-.4) -- (1,.5) -- (.6,1.4);
	\draw (.9,-.3) node{$2$};
}
\label{eqn:IHPoly}\\
\tikz[thick, scale = .8, baseline=(current bounding box.center)]{
\draw[ultra thick] (0,0) -- (0,1);
\fill (0,0) circle (3pt);
\fill (0,1) circle (3pt);
\draw (-.4,-.2) -- (0,0) -- (.4,-.2);
\draw (-.4,1.2) -- (0,1) -- (.4,1.2);
\draw (-.5,-.4) -- (-.5,1.4);
\draw (.5,-.4) -- (.5,1.4);
\draw (.8,-.3) node{$2$};}
& - & \hspace{.2cm}
\tikz[thick, scale = .8, baseline=(current bounding box.center)]{\draw (-0.5,.9) -- (-0.3,.5) -- (-0.5,.1);
\fill (-0.3,.5) circle (3pt);
\fill (0.7,.5) circle (3pt);
\draw [ultra thick] (-0.2,.5) -- (0.8,.5);
\draw (0.9,.9) -- (0.7,.5) -- (0.9,.1);
\draw (-.6,-.4) -- (-.6,1.4);
\draw (1,-.4)  -- (1,1.4);
\draw (1.3,-.3) node{$2$};}
&= & ~
\tikz[thick, scale = .8, baseline=(current bounding box.center)]{\fill[white!70!black] (.12,.6) to 
	(-.12,.6) to [out = -60, in = 90] 
	(-.1,.5) to [out = -90, in = -120] 
	(-.12,.4) to 
	(.12,.4) to [out = 120, in = -90]
	(.1,.5) to [out = 90, in = -120]
	(.12,.6);
	\draw (-.5,-.4) -- (-.5,1.4);
	\draw (.5,-.4) -- (.5,1.4);
	\draw (.8,-.3) node{$2$};
	\draw (-.4,-.2) to [out = 60, in = -90] (-.1,.5) to [out = 90, in = -60] (-.4,1.2);
	\draw (.4,-.2) to [out = 120, in = -90] (.1,.5) to [out = 90, in = -120] (.4,1.2);}
&-& ~
\tikz[thick, scale = .8, baseline=(current bounding box.center)]{	\draw (-.8,-.4) -- (-.8,1.4);
	\fill[white!70!black] 
	(.1,.6) to (-.1,.6) to (-.1,.4) to 
	(.1,.4) to (.1,.6);
	\draw (-.7,0.9) to [out = -30, in = 180] (0,.6) to [out = 0, in = 210] (0.7,0.9);
	\draw (-.7,0.1) to [out = 30, in = 180] (0,0.4) to [out = 0, in = 150] (0.7,0.1);
	\draw (.8,-.4) -- (.8,1.4);
	\draw (1.1,-.3) node{$2$};}
	\label{eqn:IH2Factor}
\end{alignat}

\end{lemma}

\begin{proof}
Observe that the first equation in Figure \ref{figure:2factordef} remains true when each term is rotated 90 degrees, as shown below. This new relation, subtracted from the the first equation in Figure \ref{figure:2factordef} yields Equation \ref{eqn:IHPoly}.

$$\begin{tikzpicture}[thick, scale = .8]
	
\begin{scope}[yshift = -2 cm]

\begin{scope}[xshift = -0.3cm]
\draw (-0.5,.9) -- (-0.3,.5) -- (-0.5,.1);
\fill (-0.3,.5) circle (3pt);
\fill (0.7,.5) circle (3pt);
\draw [ultra thick] (-0.2,.5) -- (0.8,.5);
\draw (0.9,.9) -- (0.7,.5) -- (0.9,.1);

\draw (-.4,-.4) -- (-.8,.5) -- (-.4,1.4);
\draw (.8,-.4) -- (1.2,.5) -- (.8,1.4);
\draw (1.1,-.3) node{$2$};
\end{scope}

\draw (1.5, .5) node{$=$};


\begin{scope}[xshift = 3.3 cm]

	\draw (-.6,-.4) -- (-1,.5) -- (-.6,1.4);
	
	\fill[white!70!black] 
	(.1,.6) to (-.1,.6) to (-.1,.4) to 
	(.1,.4) to (.1,.6);

	\draw (-.7,0.9) to [out = -30, in = 180] (0,.6) to [out = 0, in = 210] (0.7,0.9);
	\draw (-.7,0.1) to [out = 30, in = 180] (0,0.4) to [out = 0, in = 150] (0.7,0.1);

	\draw (.6,-.4) -- (1,.5) -- (.6,1.4);
	\draw (.9,-.3) node{$2$};

\end{scope}


\draw (5, .5) node{$-$};
\draw (5.5,.5) node{$z$};

\begin{scope}[xshift = 6.5 cm]
	\draw (-.4,-.4) -- (-.8,.5) -- (-.4,1.4);
	\draw (.4,-.4) -- (.8,.5) -- (.4,1.4);
	\draw (.7,-.3) node{$2$};
	\draw (-.4,-.2) -- (.4,1.2);
	\draw (-.4,1.2) -- (.4,-.2);

\end{scope}

\end{scope}

\end{tikzpicture}$$

The index $|{G:M}|_2$ of $2$-factors that factor through $M$ satisfies the same IH-relation (Equation~\ref{eqn:IH2Factor}). Label the four pictures of Equation \ref{eqn:IH2Factor} from left to right by $(G_i,M_i)$ for $i=1, 2 , 3,$ and $4$. Suppose that the non-perfect matching edges  in $(G_1,M_1)$ are part of the same cycle $C_1$ in $G_1(V_1,E_1\setminus M_1)$ as in Figure~\ref{fig:IHmovesequationexample}.  The non-perfect matching edges of the second picture of Equation \ref{eqn:IH2Factor} will then be parts of two different cycles $C_2$ and $\widehat{C}_2$ in $G_2(V_2,E_2\setminus M_2)$, and so on as displayed in Figure~\ref{fig:IHmovesequationexample}. Let $C_4$ be the cycle corresponding to $C_1$ in the fourth picture. Clearly the number of cycles in $G_1\setminus M_1$ is equal to the number of cycles in $G_4 \setminus M_4$, and since $C_4$ has two fewer edges than $C_1$, both cycles have the same parity.  Therefore, by Proposition~\ref{prop:count}, $|G_1:M_1|_2 = |G_4:M_4|_2$.

\begin{figure}[H]
$$\begin{tikzpicture}[scale=0.8]

\draw (0, .5) node{$(G_1, M_1)$};
\draw (3.7, .5) node{$(G_2, M_2)$};
\draw (7.4, .5) node{$(G_3, M_3)$};
\draw (11, .5) node{$(G_4, M_4)$};

\begin{scope}[yshift = -1.5cm]
\draw[ultra thick] (0,0) -- (0,1);
\fill (0,0) circle (3pt);
\fill (0,1) circle (3pt);
\draw (-.4,-.2) -- (0,0) -- (.4,-.2);
\draw (-.4,1.2) -- (0,1) -- (.4,1.2);

\draw[dashed] (-.4,-.2) to [out = 200, in = 270] (-1.5,0.5) to [out = 90, in = 160] (-.4,1.2);
\draw[dashed] (.4,-.2) to [out = -20, in = 270] (1.5,0.5) to [out = 90, in = 20] (.4,1.2);

\begin{scope}[xshift = 3.5cm]
	\draw (-0.5,.9) -- (-0.3,.5) -- (-0.5,.1);
	\fill (-0.3,.5) circle (3pt);
	\fill (0.7,.5) circle (3pt);
	\draw [ultra thick] (-0.2,.5) -- (0.8,.5);
	\draw (0.9,.9) -- (0.7,.5) -- (0.9,.1);

	\draw[dashed] (-0.5,0.9) to [out = 100, in = 90] (-1.5,0.5) to [out = 270, in = 260] (-0.5,0.1);
	\draw[dashed] (0.9,0.9) to [out = 80, in = 90] (1.9,0.5) to [out = 270, in = -80] (0.9,0.1);

\end{scope}

\begin{scope}[xshift = 7.4cm]

	\fill[white!70!black] (.12,.6) to 
	(-.12,.6) to [out = -60, in = 90] 
	(-.1,.5) to [out = -90, in = -120] 
	(-.12,.4) to 
	(.12,.4) to [out = 120, in = -90]
	(.1,.5) to [out = 90, in = -120]
	(.12,.6);

	\draw (-.4,-.2) to [out = 60, in = -90] (-.1,.5) to [out = 90, in = -60] (-.4,1.2);
	\draw (.4,-.2) to [out = 120, in = -90] (.1,.5) to [out = 90, in = -120] (.4,1.2);

	\draw[dashed] (-.4,-.2) to [out = 200, in = 270] (-1.5,0.5) to [out = 90, in = 160] (-.4,1.2);
	\draw[dashed] (.4,-.2) to [out = -20, in = 270] (1.5,0.5) to [out = 90, in = 20] (.4,1.2);

\end{scope}

\begin{scope}[xshift = 11cm]

	\fill[white!70!black] 
	(.1,.6) to (-.1,.6) to (-.1,.4) to 
	(.1,.4) to (.1,.6);

	\draw (-.7,0.9) to [out = -30, in = 180] (0,.6) to [out = 0, in = 210] (0.7,0.9);
	\draw (-.7,0.1) to [out = 30, in = 180] (0,0.4) to [out = 0, in = 150] (0.7,0.1);

	\draw[dashed] (-0.7,0.9) to [out = 150, in = 90] (-1.7,0.5) to [out = 270, in = 240] (-0.7,0.1);
	\draw[dashed] (0.7,0.9) to [out = 30, in = 90] (1.7,0.5) to [out = 270, in = -30] (0.7,0.1);
	
\end{scope}
\end{scope}

\begin{scope}[yshift = -2.7 cm]
\draw (0, .5) node{$C_1$};
\draw (2.8, .5) node{$C_2$};
\draw (4.7, .5) node{$\widehat{C_2}$};
\draw (6.8, .5) node{$C_3$};
\draw (8.2, .5) node{$\widehat{C_3}$};
\draw (11, .5) node{$C_4$};

\end{scope}
\end{tikzpicture}$$
\caption{An $IH$-move when the vertices of the perfect matching edge in $(G_1,M_1)$ are spanned by a single cycle, $C_1$.}
\label{fig:IHmovesequationexample}
\end{figure}
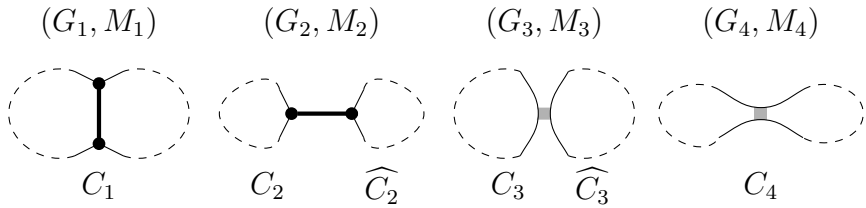

Suppose that $C_2$ and $\widehat{C}_2$ have $k$ and $\widehat{k}$ edges respectively. The corresponding cycles $C_3$ and $\widehat{C}_3$ have $k-1$ and $\widehat{k}-1$ edges respectively. If both $k$ and $\widehat{k}$ are even, then $k-1$ and $\widehat{k}-1$ are odd, implying that $|G_3:M_3|_2=0$. Also, in this case, $|G_2:M_2|_2=2|G_1:M_1|_2$.  Thus, Equation \ref{eqn:IH2Factor} is satisfied. If exactly one of $k$ and $\widehat{k}$ is odd, then all four terms of Equation \ref{eqn:IH2Factor} are zero. If both $k$ and $\widehat{k}$ are odd, then $|G_2:M_2|_2=0$ and $|G_3:M_3|_2= 2|G_1:M_1|_2$, and Equation \ref{eqn:IH2Factor} is again satisfied. 

Finally, if the non-perfect matching edges in the first picture of Equation \ref{eqn:IH2Factor}  lie in two different cycles in $G_1(V_1,E_1\setminus M_1)$, then the non-perfect matching edges in second picture of Equation \ref{eqn:IH2Factor} must be part of one cycle. Thus, the argument for this case is the same as the previous case with the pictures reversed.
\end{proof}

\begin{lemma}[Bubble Lemma]
\label{lem:bubblerelation}
The $2$-factor polynomial, when evaluated at $1$, and the index of $2$-factors that factor through $M$ both satisfy the following bubble relation:

$$\begin{tikzpicture}[scale = 0.7]
\begin{scope}[xshift= 0cm]
\draw[ultra thick] (0,-0.5) -- (0,0.25);
\draw[ultra thick] (0,0.75) -- (0,1.5);

\fill (0,-0.5) circle (3pt);
\fill (0,1.5) circle (3pt);
\fill (0,0.25) circle (3pt);
\fill (0,0.75) circle (3pt);

\draw (0,0.5) circle (0.25 cm);
\draw (-.4,-.7) -- (0,-0.5) -- (.4,-.7);
\draw (-.4,1.7) -- (0,1.5) -- (.4,1.7);
\draw (-.4,-.9) -- (-.8,.5) -- (-.4,1.9);
\draw (.4,-.9) -- (.8,.5) -- (.4,1.9);
\draw (.7,-.8) node{$2$};

\draw (1.1,0.5) node{$(1)$};
\end{scope}

\draw (1.75, .5) node {$=$};
\draw (2.25,0.5) node {$2$};

\begin{scope}[xshift= 3.25 cm]
\draw[ultra thick] (0,0) -- (0,1);
\fill (0,0) circle (3pt);
\fill (0,1) circle (3pt);
\draw (-.4,-.2) -- (0,0) -- (.4,-.2);
\draw (-.4,1.2) -- (0,1) -- (.4,1.2);
\draw (-.4,-.4) -- (-.8,.5) -- (-.4,1.4);
\draw (.4,-.4) -- (.8,.5) -- (.4,1.4);
\draw (.7,-.3) node{$2$};

\draw (1.1,0.5) node{$(1)$};

\end{scope}

\begin{scope}[yshift = -3 cm, xshift= 0cm]
\draw[ultra thick] (0,-0.5) -- (0,0.25);
\draw[ultra thick] (0,0.75) -- (0,1.5);

\fill (0,-0.5) circle (3pt);
\fill (0,1.5) circle (3pt);
\fill (0,0.25) circle (3pt);
\fill (0,0.75) circle (3pt);

\draw (0,0.5) circle (0.25 cm);
\draw (-.4,-.7) -- (0,-0.5) -- (.4,-.7);
\draw (-.4,1.7) -- (0,1.5) -- (.4,1.7);
\draw (-.5,-.9) -- (-.5,1.9);
\draw (.5,-.9) -- (.5,1.9);
\draw (.7,-.8) node{$2$};

\draw (1.75, .5) node {$=$};
\draw (2.25,0.5) node {$2$};

\begin{scope}[ xshift= 3 cm]
\draw[ultra thick] (0,0) -- (0,1);
\fill (0,0) circle (3pt);
\fill (0,1) circle (3pt);
\draw (-.4,-.2) -- (0,0) -- (.4,-.2);
\draw (-.4,1.2) -- (0,1) -- (.4,1.2);
\draw (-.5,-.4) -- (-.5,1.4);
\draw (.5,-.4) -- (.5,1.4);
\draw (.7,-.3) node{$2$};
\end{scope}
\end{scope}

\end{tikzpicture}$$
\end{lemma}

\begin{proof}
The relationship for the index of $2$-factors that factor through $M$ follows immediately from Proposition \ref{prop:count}. The proof of the bubble lemma for the $2$-factor polynomial is given by resolving the perfect matchings as in the picture below:

$$\begin{tikzpicture}[scale=0.75]
\begin{scope}[xshift= 0cm]

\draw[ultra thick] (0,-0.5) -- (0,0.25);
\draw[ultra thick] (0,0.75) -- (0,1.5);

\fill (0,-0.5) circle (3pt);
\fill (0,1.5) circle (3pt);
\fill (0,0.25) circle (3pt);
\fill (0,0.75) circle (3pt);

\draw (0,0.5) circle (0.25 cm);
\draw (-.4,-.7) -- (0,-0.5) -- (.4,-.7);
\draw (-.4,1.7) -- (0,1.5) -- (.4,1.7);
\draw (-.4,-.9) -- (-.8,.5) -- (-.4,1.9);
\draw (.4,-.9) -- (.8,.5) -- (.4,1.9);
\draw (.7,-.8) node{$2$};

\draw (1.1,0.5) node{$(1)$};
\end{scope}

\draw (1.75, .5) node {$=$};

\begin{scope}[xshift = 3.25 cm]

	\fill[white!70!black] (.11,1.35) to 
	(-.11,1.35) to [out = -60, in = 90] 
	(-.1,1.25) to [out = -90, in = -120] 
	(-.11,1.15) to 
	(.11,1.15) to [out = 120, in = -90]
	(.1,1.25) to [out = 90, in = -120]
	(.11,1.35);
	
\draw[ultra thick] (0,-0.5) -- (0,0.25);

\fill (0,-0.5) circle (3pt);
\fill (0,0.25) circle (3pt);

\draw (0,0.25) to [out = 95, in = 270] (-0.2,0.75);
\draw (-0.2,0.75) to [out = 90, in = 270] (-0.1,1.25);
\draw (-0.1,1.25) to [out = 90, in = 290] (-0.3,1.75);

\draw (0,0.25) to [out = 85, in = 270] (0.2,0.75);
\draw (0.2,0.75) to [out = 90, in = 270] (0.1,1.25);
\draw (0.1,1.25) to [out = 90, in = 250] (0.3,1.75);

\draw (-.4,-.7) -- (0,-0.5) -- (.4,-.7);
\draw (-.4,-.9) -- (-.8,.5) -- (-.4,1.9);
\draw (.4,-.9) -- (.8,.5) -- (.4,1.9);
\draw (.7,-.8) node{$2$};

\draw (1.1,0.5) node{$(1)$};
\end{scope}

\draw (5,0.5) node{$-$};

\begin{scope}[xshift = 6.5 cm]
\draw[ultra thick] (0,-0.5) -- (0,0.25);

\fill (0,-0.5) circle (3pt);
\fill (0,0.25) circle (3pt);

\draw (0,0.25) to [out = 95, in = 270] (-0.2,0.75);
\draw (-0.2,0.75) to [out = 90, in = 250] (0.3,1.75);

\draw (0,0.25) to [out = 85, in = 270] (0.2,0.75);
\draw (0.2,0.75) to [out = 90, in = 290] (-0.3,1.75);

\draw (-.4,-.7) -- (0,-0.5) -- (.4,-.7);
\draw (-.4,-.9) -- (-.8,.5) -- (-.4,1.9);
\draw (.4,-.9) -- (.8,.5) -- (.4,1.9);
\draw (.7,-.8) node{$2$};

\draw (1.1,0.5) node{$(1)$};
\end{scope}


\begin{scope}[yshift = -3 cm]

\draw (1.75, .5) node {$=$};

\begin{scope}[xshift = 3.25 cm]

	\fill[white!70!black] 
	(.11,1.15 + 0.1) 
	to (-.11,1.15 + 0.1) to (-.11,1.15 - 0.1) 
	to (.11,1.15 - 0.1) to (.11,1.15 +0.1);
	
	\fill[white!70!black] 
	(.11,-0.15 + 0.1) 
	to (-.11,-0.15 + 0.1) to (-.11,-0.15 - 0.1) 
	to (.11,-0.15 - 0.1) to (.11,0.15 +0.1);

\draw (-0.3,-0.8) 
to [out = 70, in = 270] (-0.1,-0.8 + 0.65)
to [out = 90, in = 270] (-0.2,-0.8 + 2*0.65 ) 
to [out = 90, in = 270] (-0.1,-0.8 + 3*0.65)
to [out = 90, in = 290] (-0.3,-0.8 + 4*0.65);

\draw (0.3,-0.8) 
to [out = 110, in = 270] (0.1,-0.8 + 0.65)
to [out = 90, in = 270] (0.2,-0.8 + 2*0.65 ) 
to [out = 90, in = 270] (0.1,-0.8 + 3*0.65)
to [out = 90, in = 250] (0.3,-0.8 + 4*0.65);

\draw (-.4,-.9) -- (-.8,.5) -- (-.4,1.9);
\draw (.4,-.9) -- (.8,.5) -- (.4,1.9);
\draw (.7,-.8) node{$2$};

\draw (1.1,0.5) node{$(1)$};
\end{scope}

\draw (5,0.5) node{$-$};

\begin{scope}[xshift = 6.5 cm]

	\fill[white!70!black] 
	(.11,1.15 + 0.1) 
	to (-.11,1.15 + 0.1) to (-.11,1.15 - 0.1) 
	to (.11,1.15 - 0.1) to (.11,1.15 +0.1);
	
\draw (0.3,-0.8) 
to [out = 110, in = 290] (0,-0.8 + 0.65)
to [out = 110, in = 270] (-0.2,-0.8 + 2*0.65 ) 
to [out = 90, in = 270] (-0.1,-0.8 + 3*0.65)
to [out = 90, in = 290] (-0.3,-0.8 + 4*0.65);

\draw (-0.3,-0.8) 
to [out = 70, in = 250] (0,-0.8 + 0.65)
to [out = 70, in = 270] (0.2,-0.8 + 2*0.65 ) 
to [out = 90, in = 270] (0.1,-0.8 + 3*0.65)
to [out = 90, in = 250] (0.3,-0.8 + 4*0.65);

\draw (-.4,-.9) -- (-.8,.5) -- (-.4,1.9);
\draw (.4,-.9) -- (.8,.5) -- (.4,1.9);
\draw (.7,-.8) node{$2$};

\draw (1.1,0.5) node{$(1)$};
\end{scope}

\draw (8.3,0.5) node{$-$};

\begin{scope}[xshift = 9.8 cm]

	\fill[white!70!black] 
	(.11,-0.15 + 0.1) 
	to (-.11,-0.15 + 0.1) to (-.11,-0.15 - 0.1) 
	to (.11,-0.15 - 0.1) to (.11,0.15 +0.1);
		
\draw (0.3,-0.8 + 4*0.65) 
to [out = 250, in = 70] (0,-0.8 + 3*0.65)
to [out = 70, in = 90] (-0.2,-0.8 + 2*0.65 ) 
to [out = 270, in = 90] (-0.1,-0.8 + 0.65)
to [out = 270, in = 70] (-0.3,-0.8);

\draw (-0.3,-0.8 + 4*0.65) 
to [out = -70, in = 110] (0,-0.8 + 3*0.65)
to [out = -70, in = 90] (0.2,-0.8 + 2*0.65 ) 
to [out = 270, in = 90] (0.1,-0.8 + 0.65)
to [out = 270, in = 110] (0.3,-0.8);

\draw (-.4,-.9) -- (-.8,.5) -- (-.4,1.9);
\draw (.4,-.9) -- (.8,.5) -- (.4,1.9);
\draw (.7,-.8) node{$2$};

\draw (1.1,0.5) node{$(1)$};
\end{scope}

\draw (11.5,0.5) node{$+$};

\begin{scope}[xshift = 12.9 cm]
	
\draw (0.3,-0.8) 
to [out = 110, in = 290] (0,-0.8 + 0.65)
to [out = 110, in = 270] (-0.2,-0.8 + 2*0.65 ) 
to [out = 90, in = 250] (0,-0.8 + 3*0.65)
to [out = 70, in = 250] (0.3,-0.8 + 4*0.65); 

\draw (-0.3,-0.8) 
to [out = 70, in = 250] (0,-0.8 + 0.65)
to [out = 70, in = 270] (0.2,-0.8 + 2*0.65 ) 
to [out = 90, in = 290] (0,-0.8 + 3*0.65)
to [out = 110, in = 290] (-0.3,-0.8 + 4*0.65); 

\draw (-.4,-.9) -- (-.8,.5) -- (-.4,1.9);
\draw (.4,-.9) -- (.8,.5) -- (.4,1.9);
\draw (.7,-.8) node{$2$};

\draw (1.1,0.5) node{$(1)$};
\end{scope}

\end{scope}


\begin{scope}[yshift = -5.5 cm]

\draw (1.75, .5) node {$=$};

\begin{scope}[xshift = 3.25 cm]
	\draw (-1,.5) node{$2$};
	\fill[white!70!black] (.12,.6) to 
	(-.12,.6) to [out = -60, in = 90] 
	(-.1,.5) to [out = -90, in = -120] 
	(-.12,.4) to 
	(.12,.4) to [out = 120, in = -90]
	(.1,.5) to [out = 90, in = -120]
	(.12,.6);
	\draw (-.4,-.4) -- (-.8,.5) -- (-.4,1.4);
	\draw (.4,-.4) -- (.8,.5) -- (.4,1.4);
	\draw (.7,-.3) node{$2$};
	\draw (-.4,-.2) to [out = 60, in = -90] (-.1,.5) to [out = 90, in = -60] (-.4,1.2);
	\draw (.4,-.2) to [out = 120, in = -90] (.1,.5) to [out = 90, in = -120] (.4,1.2);
	\draw (1.1,0.5) node{$(1)$};
	
\end{scope}


\draw (5, .5) node{$-$};
\draw (5.5,.5) node{$2$};

\begin{scope}[xshift = 6.5 cm]
	\draw (-.4,-.4) -- (-.8,.5) -- (-.4,1.4);
	\draw (.4,-.4) -- (.8,.5) -- (.4,1.4);
	\draw (.7,-.3) node{$2$};
	\draw (-.4,-.2) -- (.4,1.2);
	\draw (-.4,1.2) -- (.4,-.2);
	\draw (1.1,0.5) node{$(1)$};
\end{scope}

\end{scope}

\begin{scope}[yshift = -7.5 cm, xshift= 3.25 cm]
\draw (-1.5, .5) node {$=$};

	\draw (-1,.5) node{$2$};
\draw[ultra thick] (0,0) -- (0,1);
\fill (0,0) circle (3pt);
\fill (0,1) circle (3pt);
\draw (-.4,-.2) -- (0,0) -- (.4,-.2);
\draw (-.4,1.2) -- (0,1) -- (.4,1.2);
\draw (-.4,-.4) -- (-.8,.5) -- (-.4,1.4);
\draw (.4,-.4) -- (.8,.5) -- (.4,1.4);
\draw (.7,-.3) node{$2$};

\draw (1.1,0.5) node{$(1)$};

\end{scope}

\end{tikzpicture}$$

\end{proof}

\begin{remark}
Note that the bubble relation in Lemma \ref{lem:bubblerelation} is not true of the polynomial itself, just the value of the polynomial when evaluated at $1$.  
\end{remark}

\begin{lemma}
\label{lem:bridgelemma}
Let $G$ be a connected planar trivalent graph with a perfect matching $M$.  If $G$ contains a bridge, then $\langle G : M\rangle_2(1) = 0$, and $|{G:M}|_2 = 0$.  
\end{lemma}

\begin{proof}
The result  for $|{G:M}|_2 $ follows from Proposition \ref{prop:count}. For the $2$-factor polynomial, first note that if $G$ has a bridge, then it must be a perfect matching edge (cf. \cite{BaldCohomology}). Resolving this perfect matching edge shows that the $2$-factor polynomial has a $(z-1)$ factor, and therefore $\langle {G:M}\rangle_2(1)=0$ (again, see \cite{BaldCohomology}).
\end{proof}

\begin{figure}[h]
$$\begin{tikzpicture}[thick, scale = 0.8]

\fill (0,0) circle (3pt);
\fill (0,1) circle (3pt);
\draw [ultra thick] (0,0) -- (0,1);
\draw [thin] (0,0) to [out = 135, in =-90] (-.25,.5) to [out = 90, in = 225] (0,1);

\draw[->] (1,.5) -- (2,.5);

\draw [thin] (2.95,.5) circle (.3cm);
\draw [ultra thick] (3.25,.5) -- (4.25,.5);
\fill (3.25,.5) circle (3pt);
\fill (4.25,.5) circle (3pt);


\begin{scope}[yshift = -1.5cm]
	\fill (0,0) circle (3pt);
	\fill (0,1) circle (3pt);
	\draw [ultra thick] (0,0) -- (0,1);
	\draw [thin] (0,0) -- (-.72,.5) --  (0,1);
	\fill (-.72,.5) circle (3pt);
	\draw[ultra thick] (-.72,.5)  -- (-1.1,.5);

	\draw[->] (1,.5) -- (2,.5);

	\draw [thin] (3.2,.5) ellipse(.3cm and .2cm);
	\draw [ultra thick] (3.5,.5) -- (4.5,.5);
	\fill (3.5,.5) circle (3pt);
	\fill (4.5,.5) circle (3pt);
	\fill (2.9, .5) circle (3pt);
	\draw [ultra thick] (2.9,.5) -- (2.4,.5);
\end{scope}


\begin{scope}[yshift = -3cm]
	\fill (0,0) circle (3pt);
	\fill (0,1) circle (3pt);
	\draw [ultra thick] (0,0) -- (0,1);
	\draw [thin] (0,0) -- (-1,0) -- (-1,1) -- (0,1);
	\fill (-1,0) circle (3pt);
	\fill (-1,1) circle (3pt);
	\draw[ultra thick] (-1,0) -- (-1.2,-.2);
	\draw[ultra thick] (-1,1) -- (-1.2, 1.2);

	\draw[->] (1,.5) -- (2,.5);

	\draw [thin] (3.5,.5) -- (2.8,1) -- (2.8,0) -- (3.5,.5);
	\draw [ultra thick] (3.5,.5) -- (4.5,.5);
	\fill (3.5,.5) circle (3pt);
	\fill (4.5,.5) circle (3pt);
	\fill (2.8,1) circle (3pt);
	\fill (2.8,0) circle (3pt);
	\draw [ultra thick] (2.8,1) -- (2.6,1.15);
	\draw [ultra thick] (2.8,0) -- (2.6,-.15);
\end{scope}

\begin{scope}[xshift = 6cm]

	\draw[ultra thick] (0,0) -- (0,1);
	\draw[ultra thick] (1,0) -- (1,1);
	\draw [thin] (-.3,1) -- (1.3,1);
	\draw [thin] (-.3,0) -- (1.3,0);
	\fill (0,0) circle (3pt);
	\fill (0,1) circle (3pt);
	\fill (1,0) circle (3pt);
	\fill (1,1) circle (3pt);

	\draw[->] (2,.5) -- (3,.5); 
	
	\draw [thin] (4.5,.5) ellipse(.3cm and .2cm);
	\draw [ultra thick] (4.8,.5) -- (5.5,.5);
	\fill (4.2,.5) circle (3pt);
	\fill (4.8,.5) circle (3pt);
	\fill (3.5, .5) circle (3pt);
	\fill (5.5,.5) circle (3pt);
	\draw [ultra thick] (4.2,.5) -- (3.5,.5);

\end{scope}


\begin{scope}[xshift = 6cm, yshift = -1.5cm]

	\draw[ultra thick] (0,0) -- (0,1);
	\draw[ultra thick] (1,0) -- (1,1);
	\draw [thin] (-.3,1) -- (1.3,1);
	\draw [thin] (-.3,0) -- (1.3,0);
	\fill (0,0) circle (3pt);
	\fill (0,1) circle (3pt);
	\fill (1,0) circle (3pt);
	\fill (1,1) circle (3pt);
	\fill (.5,1) circle (3pt);
	\draw[ultra thick] (.5,1) -- (.5,1.4);

	\draw[->] (2,.5) -- (3,.5); 
	
	\fill (4,0) circle (3pt);
	\fill (5,0) circle (3pt);
	\fill (4.5,1) circle (3pt);
	\draw[ultra thick] (4,0) -- (3.6,0);
	\draw[ultra thick] (5,0) -- (5.4,0);
	\draw[ultra thick] (4.5,1) -- (4.5,1.4);
	\draw [thin] (4,0) -- (5,0) -- (4.5,1) -- (4,0);

\end{scope}


\begin{scope}[xshift = 6cm, yshift = -3cm]

	\draw[ultra thick] (.25,.9) -- (.75,.9);
	\draw[ultra thick] (.25,.1) -- (0,.5);
	\draw[ultra thick] (.75,.1) -- (1,.5);
	\draw [thin] (1.3,.5) -- (1,.5) -- (.75,.9) -- (1,1.2);
	\draw [thin] (0, 1.2) -- (.25,.9) -- (0,.5) -- (-.3,.5);
	\draw [thin] (0,0) -- (.25,.1) -- (.75,.1) -- (1,0);
	\fill (1,.5) circle (3pt);
	\fill (.75, .9) circle (3pt);
	\fill (.25,.9) circle (3pt);
	\fill (0,.5) circle (3pt);
	\fill (.25,.1) circle (3pt);
	\fill (.75,.1)circle (3pt);
	
	\draw[->] (2,.5) -- (3,.5); 
	
	\fill (4,0) circle (3pt);
	\fill (5,0) circle (3pt);
	\fill (4.5,1) circle (3pt);
	\draw[ultra thick] (4,0) -- (3.6,0);
	\draw[ultra thick] (5,0) -- (5.4,0);
	\draw[ultra thick] (4.5,1) -- (4.5,1.4);
	\draw [thin] (4,0) -- (5,0) -- (4.5,1) -- (4,0);

\end{scope}

\end{tikzpicture}$$
\caption{Local configurations of the pair $(G,M)$ that may appear when ${G(V,E\setminus M)}$ has one cycle.}
\label{figure:configs}
\end{figure}
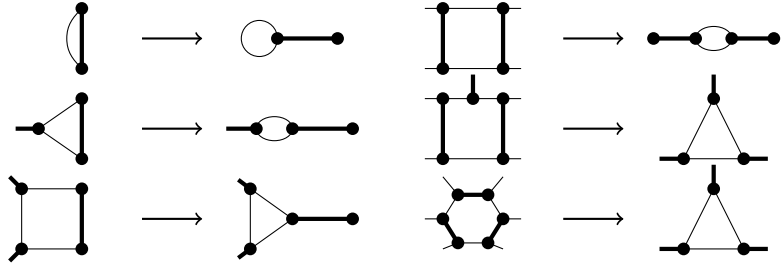

\begin{lemma}
\label{lemma:DiskLemma}
Let the pair $(G,M)$ be a planar trivalent graph $G$ with perfect matching $M$.  If  ${G(V,E\setminus M})$ is a single cycle, then $G$ must contain one of the local configurations appearing to the left of each arrow in Figure \ref{figure:configs}.
\end{lemma}

\begin{proof}
Suppose that $M$ has $k$ edges, and that $G$ has $n$ vertices ($n=2k$). Since ${G(V,{E\setminus M})}$ is a cycle, it has an inside and an outside. Either the inside or outside of $C$ contains at least $k/2$ edges of the perfect matching $M$. Without loss of generality, assume the inside of $C$ contains at least $k/2$ edges of $M$. Call this inside region $D$.  Since at least $k/2$ edges of $M$ lie in $D$, those edges split $D$ into at least $1 + k/2$ disk faces.  The boundary of each of these disks will contain some number of edges of $M$, and some number of vertices that do not belong to any edge of $M$.  If such a disk has $m$ edges from $M$ and $\ell$ vertices not incident to any edge of $M$ on the inside of the cycle $C$, then we label that disk $D^\ell_m$. The configurations in Figure \ref{figure:configs} are $D_1^0, D_1^1, D_1^2, D_2^0, D_2^1$ and $D_3^0$. By way of contradiction, assume no such configuration appears in $D$.

We form another graph $T$ by placing a vertex in each face of $D$ and joining two vertices when their associated faces are separated by an edge of $M$. The graph $T$ is a tree, and the degree of each vertex in $T$ is the number of perfect matching edges $M$ in the boundary of that disk face. Let $\delta_m$ denote the number of faces of the form $D_m^\ell$ for any $\ell$. Hence $\delta_m$ is the number of degree $m$ vertices in the tree $T$. A standard way to count leaves in a tree results in the following formula:
$$\delta_1 = 2 + \sum_{m=3}^\infty (m-2)\delta_m.$$

Define $\delta_{\ge 4} = \sum_{m=4}^\infty \delta_m$. The above equality implies that $\delta_1 \geq 2 + \delta_3 + 2\delta_{\ge 4},$ or equivalently that $\delta_{\ge 4} \leq  \frac{1}{2}\left(\delta_1 - \delta_3\right)-1$. Therefore
 $$ \frac{k}{2}+1 \leq \delta_1 + \delta_2 + \delta_3 + \delta_{\ge 4}  \leq   \frac{3\delta_1}{2} + \delta_2 + \frac{\delta_3 }{2}-1,$$
 or
 $$k \leq 3\delta_1 +2\delta_2 + \delta_3 - 4.$$

Since the faces in Figure \ref{figure:configs} are assumed not to appear in $C$, each disk of the form $D_1^{\ell}$ must have $\ell\geq 3$, each disk of the form $D_2^{\ell}$ must have $\ell\geq 2$, and each disk of the form $D_3^\ell$ must have $\ell\geq 1$. There are also at most $k$ vertices of $G$ not incident to an edge of the matching $M$ inside $C$. By counting the number of vertices of $G$ not incident to an edge of the matching $M$ inside $C$, we obtain the inequality 
$$k \geq 3\delta_1 + 2\delta_2 + \delta_3.$$
Therefore
$$k \geq 3\delta_1 + 2\delta_2 + \delta_3>3\delta_1 +2\delta_2 + \delta_3 - 4 \geq k.$$
Having arrived at a contradiction, we conclude that some face from Figure \ref{figure:configs} appears in $C$.
\end{proof}

\begin{lemma}
\label{lem:IHMaxCycle}
Let $G$ be a connected bridgeless planar trivalent  graph with perfect matching $M$. There exists a sequence of IH-moves transforming the pair $(G,M)$ into the pair $(G',M')$ where $G'$ is a planar trivalent bridgeless graph and $M'$ is a perfect matching of $G'$ such that $G'(V',E'\setminus M')$ contains a cycle of length at most three.
\end{lemma}
\begin{proof}
An IH-move either merges or splits two cycles of $G(V,E\setminus M)$. In a connected graph where $G(V,E\setminus M)$ contains more than one cycle, one can perform an IH-move that merges two cycles of $G(V,E\setminus M)$. Thus there is a sequence of IH-moves transforming $(G,M)$ into the pair $(G_1,M_1)$ such that $G_1(V_1,E_1\setminus M_1)$ is a single cycle $C$. 

Lemma \ref{lemma:DiskLemma} implies that a face as in Figure \ref{figure:configs} must appear either inside or outside of the cycle $C$ in $G_1$. Performing the IH-move(s) in Figure \ref{figure:configs} to the appropriate edge(s) of $M_1$ results in a pair $(G',M')$ such that $G'(V',E'\setminus M')$ contains a cycle of length at most three.
\end{proof}

\begin{lemma}
\label{lem:TriangleLemma}
Let $G$ be a planar trivalent graph $G$ with a perfect matching $M$.  If $G\setminus M$ contains a $3$-cycle, then $\langle{G:M} \rangle_2(1)  = 0$.
\end{lemma}

\begin{proof}
Assume that $M$ has $k$ edges and $G$ has $n$ vertices ($n=2k$). Observe that if 2 vertices of the cycle are also connected by a perfect matching edge, then the third vertex must be an endpoint of a perfect matching edge that is a bridge, and therefore, by Lemma~\ref{lem:bridgelemma}, we are done. Thus, we may assume that the three perfect matching edges incident with the $3$-cycle are distinct.  For the same reason, we may assume that all three perfect matching edges are all outside of the $3$-cycle or are all inside the $3$-cycle.  Thus, we may assume, without loss of generality,  that the $3$-cycle, with distinct perfect matching edges at each vertex, appears as in Figure~\ref{fig:3cycle}.

\begin{figure}[H]
$$\begin{tikzpicture}[scale=.8]

\draw[dashed] (0,-0.18) circle (1 cm);
\draw[ultra thick] (-0.5,-0.5) -- (-0.8,-0.8);
\draw[ultra thick] (0.5,-0.5) -- (0.8,-0.8);
\draw[ultra thick] (0,0.866-0.5) -- (0,0.866-0.05);

\fill (-0.5,-0.5) circle (3pt);
\fill (0.5,-0.5) circle (3pt);
\fill (0,0.866-0.5) circle (3pt);
\draw (-0.5,-0.5) -- (0.5,-0.5) -- (0,0.866-0.5) -- cycle;

\end{tikzpicture}$$
\caption{A $3$-cycle.}
\label{fig:3cycle}
\end{figure}
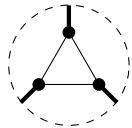

Because the polynomial is calculated recursively, the cube of resolutions for $\langle {G:M} \rangle_2(1)$ may be decomposed as follows.  First, resolve every perfect matching edge that is not incident to the $3$-cycle to get $2^{k-3}$ ``prestates."  Each prestate $P$ will consist of a number of immersed circles and a graph $G'$ whose vertices are spanned only by the three perfect matching edges of the $3$-cycle.  Note: the graph $G'$ may be immersed away from the $3$-cycle and the three perfect matching edges.  If there are $\ell$ immersed circles in the prestate $P$, then by the relations in Figure~\ref{figure:2factordef}, $$\langle P \rangle_2(1) =  2^\ell \cdot \left(\langle G' \rangle_2(1)\right).$$
We will show that in all possible cases for $G'$, $\langle G' \rangle_2(1) = 0$.  Since $\langle {G:M} \rangle_2(1)$ is made up as sums of terms of the form $\pm \langle P \rangle_2(1)$, and each of these terms is zero,  $\langle{G:M}\rangle_2(1)=0$.

Resolving all three perfect matching edges of the $3$-cycle using all open-resolutions $)($ produces the picture shown in Figure \ref{fig:3cycleres}. Thus, the ``all-open-resolution'' of the graph $G'$ of a prestate $P$  corresponds to a choice of connecting the six endpoints of Figure~\ref{fig:3cycleres} with three arcs.  

\begin{figure}[h]
$$\begin{tikzpicture}[scale=.8]

\draw[dashed] (0,-0.18) circle (1 cm);
\draw (0.15,0.866-0.05) to [out = -90, in = 135] (0.8 + 0.08,-0.8 + 0.15);
\draw (-0.15,0.866-0.05) to [out = -90, in = 45] (-0.8 - 0.08,-0.8 + 0.15);
\draw (-0.8 + 0.095,-0.8 - 0.1) to [out = 45, in = 135] (0.8 - 0.095,-0.8 - 0.1);

\end{tikzpicture}$$
\caption{Resolving $G'$ using three open-resolutions $)($ on the perfect matching edges.}
\label{fig:3cycleres}
\end{figure}
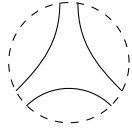

Up to symmetry of the disk, there are seven ways to connect these six endpoints up to form a set of immersed circles.  These seven configurations are shown in Figure \ref{fig:3cyclecases}. (The actual all-open-resolution of a particular $G'$ may have more circle intersections, of course, but these extra intersections do not effect any of the calculations.)  Thus, to prove $\langle G'\rangle_2(1)=0$ for all possible graphs $G'$ of the prestates, we need only show that it is zero for these seven configurations.

\begin{figure}[h]
$$\begin{tikzpicture}[scale = 0.5]

\begin{scope}[xshift = 1.5 cm]
\draw[dashed] (0,-0.18) circle (1 cm);
\draw (0.15,0.866-0.05) to [out = -90, in = 135] (0.8 + 0.08,-0.8 + 0.15);
\draw (-0.15,0.866-0.05) to [out = -90, in = 45] (-0.8 - 0.08,-0.8 + 0.15);
\draw (-0.8 + 0.095,-0.8 - 0.1) to [out = 45, in = 135] (0.8 - 0.095,-0.8 - 0.1);
\draw (0.15,0.866-0.05) to [out = 90, in = 0] (0, 1.2) to [out = 180, in = 90] (-0.15,0.866-0.05);
\draw (0.8 + 0.08,-0.8 + 0.15) to [out = -45, in = 45]  (0.8 + 0.2,-0.8 - 0.2) to [out = 225, in = -45]  (0.8 - 0.095,-0.8 - 0.1);
\draw (-0.8 - 0.08,-0.8 + 0.15) to [out = 225, in = 135] (-0.8 - 0.2,-0.8 - 0.2) to [out = -45, in = 225] (-0.8 + 0.095,-0.8 - 0.1);
\end{scope}

\begin{scope}[xshift = 4.5 cm]
\draw[dashed] (0,-0.18) circle (1 cm);
\draw (0.15,0.866-0.05) to [out = -90, in = 135] (0.8 + 0.08,-0.8 + 0.15);
\draw (-0.15,0.866-0.05) to [out = -90, in = 45] (-0.8 - 0.08,-0.8 + 0.15);
\draw (-0.8 + 0.095,-0.8 - 0.1) to [out = 45, in = 135] (0.8 - 0.095,-0.8 - 0.1);
\draw (0.15,0.866-0.05) to [out = 90, in = 0] (0, 1.2) to [out = 180, in = 90] (-0.15,0.866-0.05);
\draw (0.8 + 0.08,-0.8 + 0.15) to [out = -45, in = 0] (0, -1.75) to [out = 180, in = 225] (-0.8 - 0.08,-0.8 + 0.15) ;
\draw (0.8 - 0.095,-0.8 - 0.1) to [out = -45, in = 0] (0, -1.5) to [out = 180, in = 225] (-0.8 + 0.095,-0.8 - 0.1);
\end{scope}

\begin{scope}[xshift = 7.5 cm]
\draw[dashed] (0,-0.18) circle (1 cm);
\draw (0.15,0.866-0.05) to [out = -90, in = 135] (0.8 + 0.08,-0.8 + 0.15);
\draw (-0.15,0.866-0.05) to [out = -90, in = 45] (-0.8 - 0.08,-0.8 + 0.15);
\draw (-0.8 + 0.095,-0.8 - 0.1) to [out = 45, in = 135] (0.8 - 0.095,-0.8 - 0.1);
\draw (0.15,0.866-0.05) to [out = 90, in = 0] (0, 1.2) to [out = 180, in = 90] (-0.15,0.866-0.05);
\draw (0.8 + 0.08,-0.8 + 0.15) to [out = -45, in = -20] (0, -1.5) to [out = 160, in = 225] (-0.8 + 0.095,-0.8 - 0.1);
\draw (0.8 - 0.095,-0.8 - 0.1) to [out = -45, in = 20] (0, -1.5) to [out = 200, in = 225]  (-0.8 - 0.08,-0.8 + 0.15) ;
\end{scope}

\begin{scope}[yshift = -3.2 cm]

\draw[dashed] (0,-0.18) circle (1 cm);
\draw (0.15,0.866-0.05) to [out = -90, in = 135] (0.8 + 0.08,-0.8 + 0.15);
\draw (-0.15,0.866-0.05) to [out = -90, in = 45] (-0.8 - 0.08,-0.8 + 0.15);
\draw (-0.8 + 0.095,-0.8 - 0.1) to [out = 45, in = 135] (0.8 - 0.095,-0.8 - 0.1);
\draw (0.15,0.866-0.05) to [out = 90, in = 135] (1.2,0.7) to [out = -45, in = -45] (0.8 + 0.08,-0.8 + 0.15);
\draw (-0.15,0.866-0.05) to [out = 90, in = 45] (-1.20, 0.7) to [out = 225, in = 225] (-0.8 - 0.08,-0.8 + 0.15);
\draw (0.8 - 0.095,-0.8 - 0.1) to [out = -45, in = 0] (0, -1.6) to [out = 180, in = 225]  (-0.8 + 0.095,-0.8 - 0.1);

\begin{scope}[xshift = 3 cm]
\draw[dashed] (0,-0.18) circle (1 cm);
\draw (0.15,0.866-0.05) to [out = -90, in = 135] (0.8 + 0.08,-0.8 + 0.15);
\draw (-0.15,0.866-0.05) to [out = -90, in = 45] (-0.8 - 0.08,-0.8 + 0.15);
\draw (-0.8 + 0.095,-0.8 - 0.1) to [out = 45, in = 135] (0.8 - 0.095,-0.8 - 0.1);
\draw (0.15,0.866-0.05) to [out = 90, in = 135] (1.2,0.7) to [out = -45, in = -45] (0.8 - 0.095,-0.8 - 0.1);
\draw (-0.15,0.866-0.05) to [out = 90, in = 45] (-1.20, 0.7) to [out = 225, in = 225] (-0.8 - 0.08,-0.8 + 0.15);
\draw (0.8 + 0.08,-0.8 + 0.15) to [out = -45, in = 0] (0, -1.6) to [out = 180, in = 225]  (-0.8 + 0.095,-0.8 - 0.1);
\end{scope}

\begin{scope}[xshift = 6 cm]
\draw[dashed] (0,-0.18) circle (1 cm);
\draw (0.15,0.866-0.05) to [out = -90, in = 135] (0.8 + 0.08,-0.8 + 0.15);
\draw (-0.15,0.866-0.05) to [out = -90, in = 45] (-0.8 - 0.08,-0.8 + 0.15);
\draw (-0.8 + 0.095,-0.8 - 0.1) to [out = 45, in = 135] (0.8 - 0.095,-0.8 - 0.1);
\draw (0.15,0.866-0.05) to [out = 90, in = 135] (1.2,0.7) to [out = -45, in = -45] (0.8 - 0.095,-0.8 - 0.1);
\draw (-0.15,0.866-0.05) to [out = 90, in = 45] (-1.20, 0.7) to [out = 225, in = 225]  (-0.8 + 0.095,-0.8 - 0.1);
\draw (0.8 + 0.08,-0.8 + 0.15) to [out = -45, in = 0] (0, -1.6) to [out = 180, in = 225] (-0.8 - 0.08,-0.8 + 0.15);
\end{scope}

\begin{scope}[xshift = 9 cm]
\draw[dashed] (0,-0.18) circle (1 cm);
\draw (0.15,0.866-0.05) to [out = -90, in = 135] (0.8 + 0.08,-0.8 + 0.15);
\draw (-0.15,0.866-0.05) to [out = -90, in = 45] (-0.8 - 0.08,-0.8 + 0.15);
\draw (-0.8 + 0.095,-0.8 - 0.1) to [out = 45, in = 135] (0.8 - 0.095,-0.8 - 0.1);
\draw (-0.15,0.866-0.05) to [out = 90, in = 135] (1.2,0.7) to [out = -45, in = -45] (0.8 - 0.095,-0.8 - 0.1);
\draw (0.15,0.866-0.05) to [out = 90, in = 45] (-1.20, 0.7) to [out = 225, in = 225]  (-0.8 + 0.095,-0.8 - 0.1);
\draw (0.8 + 0.08,-0.8 + 0.15) to [out = -45, in = 0] (0, -1.6) to [out = 180, in = 225] (-0.8 - 0.08,-0.8 + 0.15);

\end{scope}

\end{scope}

\end{tikzpicture}$$
\caption{The possible $G'$'s after resolving all three perfect matchings with open-resolutions.}
\label{fig:3cyclecases}
\end{figure}
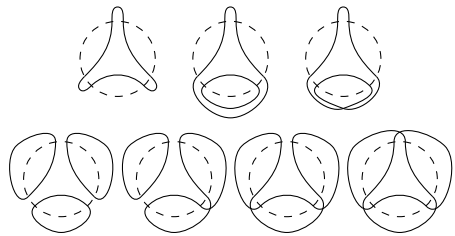

The cube of resolutions corresponding to the first all-open-resolution configuration in Figure~\ref{fig:3cyclecases} is shown in Figure \ref{fig:3cycleC1}.  For a $G'$ corresponding to this configuration, we see that $\langle G' \rangle_2(1)= 2-3\cdot 2+3\cdot 2-2=0$.  The remaining six cases are handled in the Appendix. In each case, $\langle G'\rangle_2(1)=0$, which proves the lemma.
\end{proof}

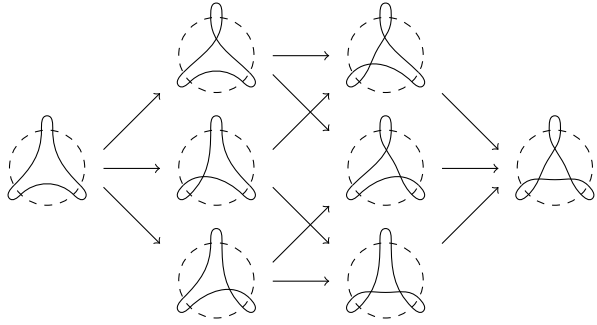
\begin{figure}[H]
$$\begin{tikzpicture}[scale = 0.5]

\begin{scope}
\draw[dashed] (0,-0.18) circle (1 cm);
\draw (0.15,0.866-0.05) to [out = -90, in = 135] (0.8 + 0.08,-0.8 + 0.15);
\draw (-0.15,0.866-0.05) to [out = -90, in = 45] (-0.8 - 0.08,-0.8 + 0.15);
\draw (-0.8 + 0.095,-0.8 - 0.1) to [out = 45, in = 135] (0.8 - 0.095,-0.8 - 0.1);

\draw (0.15,0.866-0.05) to [out = 90, in = 0] (0, 1.2) to [out = 180, in = 90] (-0.15,0.866-0.05);
\draw (0.8 + 0.08,-0.8 + 0.15) to [out = -45, in = 45]  (0.8 + 0.2,-0.8 - 0.2) to [out = 225, in = -45]  (0.8 - 0.095,-0.8 - 0.1);
\draw (-0.8 - 0.08,-0.8 + 0.15) to [out = 225, in = 135] (-0.8 - 0.2,-0.8 - 0.2) to [out = -45, in = 225] (-0.8 + 0.095,-0.8 - 0.1);
\end{scope}

\begin{scope}[xshift = 1 cm, yshift = -0.2 cm]
\draw[->] (0.5,0.5) --(2,2);
\draw[->] (0.5,0) --(2,0);
\draw[->] (0.5,-0.5) --(2,-2);
\end{scope}

\begin{scope}[xshift = 5.5 cm, yshift = -0.2 cm]
\draw[->] (0.5,0.5) --(2,2);
\draw[->] (0.5,-0.5) --(2,-2);
\draw[->] (0.5,3) --(2,3);
\draw[->] (0.5,-3) --(2,-3);
\draw[->] (0.5,-2.5) --(2,-1);
\draw[->] (0.5,2.5) --(2,1);
\end{scope}

\begin{scope}[xshift = 8.5 cm, yshift = -0.2 cm]
\draw[->] (2,2) -- (3.5,0.5);
\draw[->] (2,0) -- (3.5,0);
\draw[->] (2,-2) -- (3.5,-0.5);
\end{scope}

\begin{scope}[xshift = 4.5 cm, yshift = 3 cm]
\draw[dashed] (0,-0.18) circle (1 cm);
\draw (-0.15,0.866-0.05) to [out = -90, in = 135] (0.8 + 0.08,-0.8 + 0.15);
\draw (0.15,0.866-0.05) to [out = -90, in = 45] (-0.8 - 0.08,-0.8 + 0.15);
\draw (-0.8 + 0.095,-0.8 - 0.1) to [out = 45, in = 135] (0.8 - 0.095,-0.8 - 0.1);

\draw (0.15,0.866-0.05) to [out = 90, in = 0] (0, 1.2) to [out = 180, in = 90] (-0.15,0.866-0.05);
\draw (0.8 + 0.08,-0.8 + 0.15) to [out = -45, in = 45]  (0.8 + 0.2,-0.8 - 0.2) to [out = 225, in = -45]  (0.8 - 0.095,-0.8 - 0.1);
\draw (-0.8 - 0.08,-0.8 + 0.15) to [out = 225, in = 135] (-0.8 - 0.2,-0.8 - 0.2) to [out = -45, in = 225] (-0.8 + 0.095,-0.8 - 0.1);
\end{scope}

\begin{scope}[xshift = 4.5 cm, yshift = 0 cm]
\draw[dashed] (0,-0.18) circle (1 cm);
\draw (0.15,0.866-0.05) to [out = -90, in = 135] (0.8 + 0.08,-0.8 + 0.15);
\draw (-0.15,0.866-0.05) to [out = -90, in = 45] (-0.8 + 0.095,-0.8 - 0.1);
\draw (-0.8 - 0.08,-0.8 + 0.15) to [out = 45, in = 135] (0.8 - 0.095,-0.8 - 0.1);

\draw (0.15,0.866-0.05) to [out = 90, in = 0] (0, 1.2) to [out = 180, in = 90] (-0.15,0.866-0.05);
\draw (0.8 + 0.08,-0.8 + 0.15) to [out = -45, in = 45]  (0.8 + 0.2,-0.8 - 0.2) to [out = 225, in = -45]  (0.8 - 0.095,-0.8 - 0.1);
\draw (-0.8 - 0.08,-0.8 + 0.15) to [out = 225, in = 135] (-0.8 - 0.2,-0.8 - 0.2) to [out = -45, in = 225] (-0.8 + 0.095,-0.8 - 0.1);
\end{scope}

\begin{scope}[xshift = 4.5 cm, yshift = -3 cm]
\draw[dashed] (0,-0.18) circle (1 cm);
\draw (0.15,0.866-0.05) to [out = -90, in = 135] (0.8 - 0.095,-0.8 - 0.1);
\draw (-0.15,0.866-0.05) to [out = -90, in = 45] (-0.8 - 0.08,-0.8 + 0.15);
\draw (-0.8 + 0.095,-0.8 - 0.1) to [out = 45, in = 135] (0.8 + 0.08,-0.8 + 0.15);

\draw (0.15,0.866-0.05) to [out = 90, in = 0] (0, 1.2) to [out = 180, in = 90] (-0.15,0.866-0.05);
\draw (0.8 + 0.08,-0.8 + 0.15) to [out = -45, in = 45]  (0.8 + 0.2,-0.8 - 0.2) to [out = 225, in = -45]  (0.8 - 0.095,-0.8 - 0.1);
\draw (-0.8 - 0.08,-0.8 + 0.15) to [out = 225, in = 135] (-0.8 - 0.2,-0.8 - 0.2) to [out = -45, in = 225] (-0.8 + 0.095,-0.8 - 0.1);
\end{scope}

\begin{scope}[xshift = 4.5 cm]
\begin{scope}[xshift = 4.5 cm, yshift = 3 cm]
\draw[dashed] (0,-0.18) circle (1 cm);
\draw (-0.15,0.866-0.05) to [out = -90, in = 135] (0.8 + 0.08,-0.8 + 0.15);
\draw (0.15,0.866-0.05) to [out = -90,  in = 60] (-0.2, 0) to [out = 240, in = 45] (-0.8 + 0.095,-0.8 - 0.1);
\draw (-0.8 - 0.08,-0.8 + 0.15) to [out = 45, in = 135] (0.8 - 0.095,-0.8 - 0.1);

\draw (0.15,0.866-0.05) to [out = 90, in = 0] (0, 1.2) to [out = 180, in = 90] (-0.15,0.866-0.05);
\draw (0.8 + 0.08,-0.8 + 0.15) to [out = -45, in = 45]  (0.8 + 0.2,-0.8 - 0.2) to [out = 225, in = -45]  (0.8 - 0.095,-0.8 - 0.1);
\draw (-0.8 - 0.08,-0.8 + 0.15) to [out = 225, in = 135] (-0.8 - 0.2,-0.8 - 0.2) to [out = -45, in = 225] (-0.8 + 0.095,-0.8 - 0.1);
\end{scope}

\begin{scope}[xshift = 4.5 cm, yshift = 0 cm]
\draw[dashed] (0,-0.18) circle (1 cm);
\draw (-0.15,0.866-0.05) to [out = -90, in = 120] (0.2, 0) to [out = -60, in = 135] (0.8 - 0.095,-0.8 - 0.1);
\draw (0.15,0.866-0.05) to [out = -90, in = 45] (-0.8 - 0.08,-0.8 + 0.15);
\draw (-0.8 + 0.095,-0.8 - 0.1) to [out = 45, in = 135] (0.8 + 0.08,-0.8 + 0.15);

\draw (0.15,0.866-0.05) to [out = 90, in = 0] (0, 1.2) to [out = 180, in = 90] (-0.15,0.866-0.05);
\draw (0.8 + 0.08,-0.8 + 0.15) to [out = -45, in = 45]  (0.8 + 0.2,-0.8 - 0.2) to [out = 225, in = -45]  (0.8 - 0.095,-0.8 - 0.1);
\draw (-0.8 - 0.08,-0.8 + 0.15) to [out = 225, in = 135] (-0.8 - 0.2,-0.8 - 0.2) to [out = -45, in = 225] (-0.8 + 0.095,-0.8 - 0.1);
\end{scope}

\begin{scope}[xshift = 4.5 cm, yshift = -3 cm]
\draw[dashed] (0,-0.18) circle (1 cm);
\draw (0.15,0.866-0.05) to [out = -90, in = 135] (0.8 - 0.095,-0.8 - 0.1);
\draw (-0.15,0.866-0.05) to [out = -90, in = 45] (-0.8 + 0.095,-0.8 - 0.1);
\draw (-0.8 - 0.08,-0.8 + 0.15) to [out = 45, in = 180] (0, -0.5) to [out = 0, in = 135] (0.8 + 0.08,-0.8 + 0.15);

\draw (0.15,0.866-0.05) to [out = 90, in = 0] (0, 1.2) to [out = 180, in = 90] (-0.15,0.866-0.05);
\draw (0.8 + 0.08,-0.8 + 0.15) to [out = -45, in = 45]  (0.8 + 0.2,-0.8 - 0.2) to [out = 225, in = -45]  (0.8 - 0.095,-0.8 - 0.1);
\draw (-0.8 - 0.08,-0.8 + 0.15) to [out = 225, in = 135] (-0.8 - 0.2,-0.8 - 0.2) to [out = -45, in = 225] (-0.8 + 0.095,-0.8 - 0.1);
\end{scope}

\end{scope}

\begin{scope}[xshift = 13.5 cm]
\begin{scope}
\draw[dashed] (0,-0.18) circle (1 cm);
\draw (-0.15,0.866-0.05) to [out = -90, in = 120] (0.2, 0) to [out = -60, in = 135] (0.8 - 0.095,-0.8 - 0.1);
\draw (0.15,0.866-0.05) to [out = -90,  in = 60] (-0.2, 0) to [out = 240, in = 45] (-0.8 + 0.095,-0.8 - 0.1);
\draw (-0.8 - 0.08,-0.8 + 0.15) to [out = 45, in = 180] (0, -0.5) to [out = 0, in = 135] (0.8 + 0.08,-0.8 + 0.15);

\draw (0.15,0.866-0.05) to [out = 90, in = 0] (0, 1.2) to [out = 180, in = 90] (-0.15,0.866-0.05);
\draw (0.8 + 0.08,-0.8 + 0.15) to [out = -45, in = 45]  (0.8 + 0.2,-0.8 - 0.2) to [out = 225, in = -45]  (0.8 - 0.095,-0.8 - 0.1);
\draw (-0.8 - 0.08,-0.8 + 0.15) to [out = 225, in = 135] (-0.8 - 0.2,-0.8 - 0.2) to [out = -45, in = 225] (-0.8 + 0.095,-0.8 - 0.1);
\end{scope}

\end{scope}

\end{tikzpicture}$$
\caption{For any $G'$ corresponding to the first configuration in Figure~\ref{fig:3cyclecases}, $\langle G' \rangle_2(1)=2 - 3 \cdot 2 + 3 \cdot 2 - 2 = 0$.}
\label{fig:3cycleC1}
\end{figure}

\section{Proof of Theorem \ref{thm:MainTheorem}}

Now we are ready to prove Thoerem~\ref{thm:MainTheorem}. Because both $\langle G:M\rangle_2(z)$ and $|G:M|_2$ are multiplicative under disjoint unions, we need only prove Theorem~\ref{thm:MainTheorem} for the case when $G$ is a connected graph.  Furthermore, by Lemma~\ref{lem:bridgelemma}, the theorem is trivially true if $G$ has a bridge.  Therefore we need only prove the theorem for connected, bridgeless planar trivalent graphs.

The proof proceeds by induction on the number of edges in the perfect matching $M$.  If the perfect matching has a single edge, then $(G,M)$ is the $\theta$ graph, and the theorem is  verified by inspection.  

Assume the theorem is true for  any planar trivalent graph with a perfect matching having $n-1$ edges.  Let $(G,M)$ be a connected, bridgeless planar trivalent graph with $n$ perfect matching edges.  By Lemma \ref{lem:IHMaxCycle}, the graph is IH-equivalent to $(G',M')$, where $G'(V',E'\setminus M')$ has a cycle of length at most $3$.  Lemma \ref{lem:IHLemma}, together with the inductive hypothesis, implies that we need only verify that the result holds for $(G', M')$. 

If $G'(V',E'\setminus M')$ has a cycle of length 3, then Proposition \ref{prop:count}  and Lemma \ref{lem:TriangleLemma} imply the result.  If $G'(V',E'\setminus M')$ has a cycle of length 1, then $G'$ has a bridge and Lemma~\ref{lem:bridgelemma} implies the result.  Lastly, if $G'(V',E'\setminus M')$ has no cycle of length 1 or 3, but does have a cycle of length 2, then Lemma \ref{lem:bubblerelation} (the Bubble Lemma), together with the inductive hypothesis, implies the result.  \qed

\section{Concluding Remarks}

In working with the $2$-factor polynomial, the authors noticed a close relationship between planar trivalent graphs with perfect matchings and virtual links, which will be explored in \cite{BKR}. While Theorem \ref{thm:MainTheorem} was first proved in this paper, in the context of planar trivalent graphs, this relationship suggested that an equivalent result might hold for the Jones polynomial of a virtual link.  In a later paper (cf. Theorem 5 of \cite{BaldKauffMc}), we prove the corresponding theorem in virtual link theory.

The graph-theoretic proof presented in this paper is distinctly different than the knot-theoretic proof of \cite{BaldKauffMc}.  We feel this proof has value from a graph theory perspective.  At the heart of our argument, we are essentially inducting on the number cycles in $G\setminus M$, which reduces the problem to studying the base-case of a Hamiltonian graph, i.e., a graph $G'$ and perfect matching $M'$ with the same number of vertices and edges as $(G,M)$ but with a single cycle for $G'\setminus M'$. We believe that this technique will be of value and interest to graph theorists when proving theorems like this one.

Finally, this paper, together with \cite{BaldCohomology}, \cite{BaldKauffMc}, and \cite{BKR}, establishes a relationship between the categorification of the $2$-factor polynomial and virtual Khovanov homology. Other fruitful relationships between graph and knot polynomials and their categorifications are well-known (cf. \cite{Thistlethwaite, KauffmanTutte, DFKLS, LM, DL}). In each of these relationships, the graph theoretic and knot theoretic viewpoints mutually inform and reinforce one another. We see this connection between the categorification of the $2$-factor polynomial and virtual Khovanov homology as a promising avenue for future research.

\appendix
\section{Appendix}
The remaining cases for Lemma \ref{lem:TriangleLemma} are presented here.

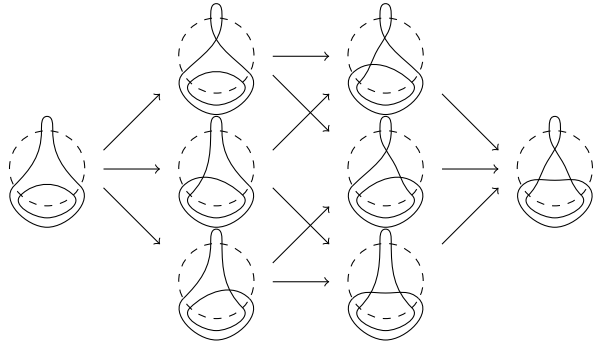
\begin{figure}[H]
$$\begin{tikzpicture}[scale = 0.5]

\begin{scope}
\draw[dashed] (0,-0.18) circle (1 cm);
\draw (0.15,0.866-0.05) to [out = -90, in = 135] (0.8 + 0.08,-0.8 + 0.15);
\draw (-0.15,0.866-0.05) to [out = -90, in = 45] (-0.8 - 0.08,-0.8 + 0.15);
\draw (-0.8 + 0.095,-0.8 - 0.1) to [out = 45, in = 135] (0.8 - 0.095,-0.8 - 0.1);

\draw (0.15,0.866-0.05) to [out = 90, in = 0] (0, 1.2) to [out = 180, in = 90] (-0.15,0.866-0.05);
\draw (0.8 + 0.08,-0.8 + 0.15) to [out = -45, in = 0] (0, -1.75) to [out = 180, in = 225] (-0.8 - 0.08,-0.8 + 0.15) ;
\draw (0.8 - 0.095,-0.8 - 0.1) to [out = -45, in = 0] (0, -1.5) to [out = 180, in = 225] (-0.8 + 0.095,-0.8 - 0.1);
\end{scope}

\begin{scope}[xshift = 1 cm, yshift = -0.2 cm]
\draw[->] (0.5,0.5) --(2,2);
\draw[->] (0.5,0) --(2,0);
\draw[->] (0.5,-0.5) --(2,-2);
\end{scope}

\begin{scope}[xshift = 5.5 cm, yshift = -0.2 cm]
\draw[->] (0.5,0.5) --(2,2);
\draw[->] (0.5,-0.5) --(2,-2);
\draw[->] (0.5,3) --(2,3);
\draw[->] (0.5,-3) --(2,-3);
\draw[->] (0.5,-2.5) --(2,-1);
\draw[->] (0.5,2.5) --(2,1);
\end{scope}

\begin{scope}[xshift = 8.5 cm, yshift = -0.2 cm]
\draw[->] (2,2) -- (3.5,0.5);
\draw[->] (2,0) -- (3.5,0);
\draw[->] (2,-2) -- (3.5,-0.5);
\end{scope}

\begin{scope}[xshift = 4.5 cm, yshift = 3 cm]
\draw[dashed] (0,-0.18) circle (1 cm);
\draw (-0.15,0.866-0.05) to [out = -90, in = 135] (0.8 + 0.08,-0.8 + 0.15);
\draw (0.15,0.866-0.05) to [out = -90, in = 45] (-0.8 - 0.08,-0.8 + 0.15);
\draw (-0.8 + 0.095,-0.8 - 0.1) to [out = 45, in = 135] (0.8 - 0.095,-0.8 - 0.1);

\draw (0.15,0.866-0.05) to [out = 90, in = 0] (0, 1.2) to [out = 180, in = 90] (-0.15,0.866-0.05);
\draw (0.8 + 0.08,-0.8 + 0.15) to [out = -45, in = 0] (0, -1.75) to [out = 180, in = 225] (-0.8 - 0.08,-0.8 + 0.15) ;
\draw (0.8 - 0.095,-0.8 - 0.1) to [out = -45, in = 0] (0, -1.5) to [out = 180, in = 225] (-0.8 + 0.095,-0.8 - 0.1);
\end{scope}

\begin{scope}[xshift = 4.5 cm, yshift = 0 cm]
\draw[dashed] (0,-0.18) circle (1 cm);
\draw (0.15,0.866-0.05) to [out = -90, in = 135] (0.8 + 0.08,-0.8 + 0.15);
\draw (-0.15,0.866-0.05) to [out = -90, in = 45] (-0.8 + 0.095,-0.8 - 0.1);
\draw (-0.8 - 0.08,-0.8 + 0.15) to [out = 45, in = 135] (0.8 - 0.095,-0.8 - 0.1);

\draw (0.15,0.866-0.05) to [out = 90, in = 0] (0, 1.2) to [out = 180, in = 90] (-0.15,0.866-0.05);
\draw (0.8 + 0.08,-0.8 + 0.15) to [out = -45, in = 0] (0, -1.75) to [out = 180, in = 225] (-0.8 - 0.08,-0.8 + 0.15) ;
\draw (0.8 - 0.095,-0.8 - 0.1) to [out = -45, in = 0] (0, -1.5) to [out = 180, in = 225] (-0.8 + 0.095,-0.8 - 0.1);
\end{scope}

\begin{scope}[xshift = 4.5 cm, yshift = -3 cm]
\draw[dashed] (0,-0.18) circle (1 cm);
\draw (0.15,0.866-0.05) to [out = -90, in = 135] (0.8 - 0.095,-0.8 - 0.1);
\draw (-0.15,0.866-0.05) to [out = -90, in = 45] (-0.8 - 0.08,-0.8 + 0.15);
\draw (-0.8 + 0.095,-0.8 - 0.1) to [out = 45, in = 135] (0.8 + 0.08,-0.8 + 0.15);

\draw (0.15,0.866-0.05) to [out = 90, in = 0] (0, 1.2) to [out = 180, in = 90] (-0.15,0.866-0.05);
\draw (0.8 + 0.08,-0.8 + 0.15) to [out = -45, in = 0] (0, -1.75) to [out = 180, in = 225] (-0.8 - 0.08,-0.8 + 0.15) ;
\draw (0.8 - 0.095,-0.8 - 0.1) to [out = -45, in = 0] (0, -1.5) to [out = 180, in = 225] (-0.8 + 0.095,-0.8 - 0.1);
\end{scope}

\begin{scope}[xshift = 4.5 cm]
\begin{scope}[xshift = 4.5 cm, yshift = 3 cm]
\draw[dashed] (0,-0.18) circle (1 cm);
\draw (-0.15,0.866-0.05) to [out = -90, in = 135] (0.8 + 0.08,-0.8 + 0.15);
\draw (0.15,0.866-0.05) to [out = -90,  in = 60] (-0.2, 0) to [out = 240, in = 45] (-0.8 + 0.095,-0.8 - 0.1);
\draw (-0.8 - 0.08,-0.8 + 0.15) to [out = 45, in = 135] (0.8 - 0.095,-0.8 - 0.1);

\draw (0.15,0.866-0.05) to [out = 90, in = 0] (0, 1.2) to [out = 180, in = 90] (-0.15,0.866-0.05);
\draw (0.8 + 0.08,-0.8 + 0.15) to [out = -45, in = 0] (0, -1.75) to [out = 180, in = 225] (-0.8 - 0.08,-0.8 + 0.15) ;
\draw (0.8 - 0.095,-0.8 - 0.1) to [out = -45, in = 0] (0, -1.5) to [out = 180, in = 225] (-0.8 + 0.095,-0.8 - 0.1);
\end{scope}

\begin{scope}[xshift = 4.5 cm, yshift = 0 cm]
\draw[dashed] (0,-0.18) circle (1 cm);
\draw (-0.15,0.866-0.05) to [out = -90, in = 120] (0.2, 0) to [out = -60, in = 135] (0.8 - 0.095,-0.8 - 0.1);
\draw (0.15,0.866-0.05) to [out = -90, in = 45] (-0.8 - 0.08,-0.8 + 0.15);
\draw (-0.8 + 0.095,-0.8 - 0.1) to [out = 45, in = 135] (0.8 + 0.08,-0.8 + 0.15);

\draw (0.15,0.866-0.05) to [out = 90, in = 0] (0, 1.2) to [out = 180, in = 90] (-0.15,0.866-0.05);
\draw (0.8 + 0.08,-0.8 + 0.15) to [out = -45, in = 0] (0, -1.75) to [out = 180, in = 225] (-0.8 - 0.08,-0.8 + 0.15) ;
\draw (0.8 - 0.095,-0.8 - 0.1) to [out = -45, in = 0] (0, -1.5) to [out = 180, in = 225] (-0.8 + 0.095,-0.8 - 0.1);
\end{scope}

\begin{scope}[xshift = 4.5 cm, yshift = -3 cm]
\draw[dashed] (0,-0.18) circle (1 cm);
\draw (0.15,0.866-0.05) to [out = -90, in = 135] (0.8 - 0.095,-0.8 - 0.1);
\draw (-0.15,0.866-0.05) to [out = -90, in = 45] (-0.8 + 0.095,-0.8 - 0.1);
\draw (-0.8 - 0.08,-0.8 + 0.15) to [out = 45, in = 180] (0, -0.5) to [out = 0, in = 135] (0.8 + 0.08,-0.8 + 0.15);

\draw (0.15,0.866-0.05) to [out = 90, in = 0] (0, 1.2) to [out = 180, in = 90] (-0.15,0.866-0.05);
\draw (0.8 + 0.08,-0.8 + 0.15) to [out = -45, in = 0] (0, -1.75) to [out = 180, in = 225] (-0.8 - 0.08,-0.8 + 0.15) ;
\draw (0.8 - 0.095,-0.8 - 0.1) to [out = -45, in = 0] (0, -1.5) to [out = 180, in = 225] (-0.8 + 0.095,-0.8 - 0.1);
\end{scope}

\end{scope}

\begin{scope}[xshift = 13.5 cm]
\begin{scope}
\draw[dashed] (0,-0.18) circle (1 cm);
\draw (-0.15,0.866-0.05) to [out = -90, in = 120] (0.2, 0) to [out = -60, in = 135] (0.8 - 0.095,-0.8 - 0.1);
\draw (0.15,0.866-0.05) to [out = -90,  in = 60] (-0.2, 0) to [out = 240, in = 45] (-0.8 + 0.095,-0.8 - 0.1);
\draw (-0.8 - 0.08,-0.8 + 0.15) to [out = 45, in = 180] (0, -0.5) to [out = 0, in = 135] (0.8 + 0.08,-0.8 + 0.15);

\draw (0.15,0.866-0.05) to [out = 90, in = 0] (0, 1.2) to [out = 180, in = 90] (-0.15,0.866-0.05);
\draw (0.8 + 0.08,-0.8 + 0.15) to [out = -45, in = 0] (0, -1.75) to [out = 180, in = 225] (-0.8 - 0.08,-0.8 + 0.15) ;
\draw (0.8 - 0.095,-0.8 - 0.1) to [out = -45, in = 0] (0, -1.5) to [out = 180, in = 225] (-0.8 + 0.095,-0.8 - 0.1);
\end{scope}

\end{scope}

\end{tikzpicture}$$
\caption{For any $G'$ corresponding to the second configuration in Figure~\ref{fig:3cyclecases}, $\langle G' \rangle_2(1)=2^2 - 2^2 - 2 \cdot 2 + 2 \cdot 2 +  2^2 - 2^2 = 0$.}
\label{fig:3cycleC2}
\end{figure}

\begin{figure}[H]
$$\begin{tikzpicture}[scale = 0.5]

\begin{scope}
\draw[dashed] (0,-0.18) circle (1 cm);
\draw (0.15,0.866-0.05) to [out = -90, in = 135] (0.8 + 0.08,-0.8 + 0.15);
\draw (-0.15,0.866-0.05) to [out = -90, in = 45] (-0.8 - 0.08,-0.8 + 0.15);
\draw (-0.8 + 0.095,-0.8 - 0.1) to [out = 45, in = 135] (0.8 - 0.095,-0.8 - 0.1);

\draw (0.15,0.866-0.05) to [out = 90, in = 0] (0, 1.2) to [out = 180, in = 90] (-0.15,0.866-0.05);
\draw (0.8 + 0.08,-0.8 + 0.15) to [out = -45, in = -20] (0, -1.5) to [out = 160, in = 225] (-0.8 + 0.095,-0.8 - 0.1);
\draw (0.8 - 0.095,-0.8 - 0.1) to [out = -45, in = 20] (0, -1.5) to [out = 200, in = 225]  (-0.8 - 0.08,-0.8 + 0.15);
\end{scope}

\begin{scope}[xshift = 1 cm, yshift = -0.2 cm]
\draw[->] (0.5,0.5) --(2,2);
\draw[->] (0.5,0) --(2,0);
\draw[->] (0.5,-0.5) --(2,-2);
\end{scope}

\begin{scope}[xshift = 5.5 cm, yshift = -0.2 cm]
\draw[->] (0.5,0.5) --(2,2);
\draw[->] (0.5,-0.5) --(2,-2);
\draw[->] (0.5,3) --(2,3);
\draw[->] (0.5,-3) --(2,-3);
\draw[->] (0.5,-2.5) --(2,-1);
\draw[->] (0.5,2.5) --(2,1);
\end{scope}

\begin{scope}[xshift = 8.5 cm, yshift = -0.2 cm]
\draw[->] (2,2) -- (3.5,0.5);
\draw[->] (2,0) -- (3.5,0);
\draw[->] (2,-2) -- (3.5,-0.5);
\end{scope}

\begin{scope}[xshift = 4.5 cm, yshift = 3 cm]
\draw[dashed] (0,-0.18) circle (1 cm);
\draw (-0.15,0.866-0.05) to [out = -90, in = 135] (0.8 + 0.08,-0.8 + 0.15);
\draw (0.15,0.866-0.05) to [out = -90, in = 45] (-0.8 - 0.08,-0.8 + 0.15);
\draw (-0.8 + 0.095,-0.8 - 0.1) to [out = 45, in = 135] (0.8 - 0.095,-0.8 - 0.1);
\draw (0.15,0.866-0.05) to [out = 90, in = 0] (0, 1.2) to [out = 180, in = 90] (-0.15,0.866-0.05);
\draw (0.8 + 0.08,-0.8 + 0.15) to [out = -45, in = -20] (0, -1.5) to [out = 160, in = 225] (-0.8 + 0.095,-0.8 - 0.1);
\draw (0.8 - 0.095,-0.8 - 0.1) to [out = -45, in = 20] (0, -1.5) to [out = 200, in = 225]  (-0.8 - 0.08,-0.8 + 0.15);
\end{scope}

\begin{scope}[xshift = 4.5 cm, yshift = 0 cm]
\draw[dashed] (0,-0.18) circle (1 cm);
\draw (0.15,0.866-0.05) to [out = -90, in = 135] (0.8 + 0.08,-0.8 + 0.15);
\draw (-0.15,0.866-0.05) to [out = -90, in = 45] (-0.8 + 0.095,-0.8 - 0.1);
\draw (-0.8 - 0.08,-0.8 + 0.15) to [out = 45, in = 135] (0.8 - 0.095,-0.8 - 0.1);
\draw (0.15,0.866-0.05) to [out = 90, in = 0] (0, 1.2) to [out = 180, in = 90] (-0.15,0.866-0.05);
\draw (0.8 + 0.08,-0.8 + 0.15) to [out = -45, in = -20] (0, -1.5) to [out = 160, in = 225] (-0.8 + 0.095,-0.8 - 0.1);
\draw (0.8 - 0.095,-0.8 - 0.1) to [out = -45, in = 20] (0, -1.5) to [out = 200, in = 225]  (-0.8 - 0.08,-0.8 + 0.15);
\end{scope}

\begin{scope}[xshift = 4.5 cm, yshift = -3 cm]
\draw[dashed] (0,-0.18) circle (1 cm);
\draw (0.15,0.866-0.05) to [out = -90, in = 135] (0.8 - 0.095,-0.8 - 0.1);
\draw (-0.15,0.866-0.05) to [out = -90, in = 45] (-0.8 - 0.08,-0.8 + 0.15);
\draw (-0.8 + 0.095,-0.8 - 0.1) to [out = 45, in = 135] (0.8 + 0.08,-0.8 + 0.15);
\draw (0.15,0.866-0.05) to [out = 90, in = 0] (0, 1.2) to [out = 180, in = 90] (-0.15,0.866-0.05);
\draw (0.8 + 0.08,-0.8 + 0.15) to [out = -45, in = -20] (0, -1.5) to [out = 160, in = 225] (-0.8 + 0.095,-0.8 - 0.1);
\draw (0.8 - 0.095,-0.8 - 0.1) to [out = -45, in = 20] (0, -1.5) to [out = 200, in = 225]  (-0.8 - 0.08,-0.8 + 0.15);
\end{scope}

\begin{scope}[xshift = 4.5 cm]
\begin{scope}[xshift = 4.5 cm, yshift = 3 cm]
\draw[dashed] (0,-0.18) circle (1 cm);
\draw (-0.15,0.866-0.05) to [out = -90, in = 135] (0.8 + 0.08,-0.8 + 0.15);
\draw (0.15,0.866-0.05) to [out = -90,  in = 60] (-0.2, 0) to [out = 240, in = 45] (-0.8 + 0.095,-0.8 - 0.1);
\draw (-0.8 - 0.08,-0.8 + 0.15) to [out = 45, in = 135] (0.8 - 0.095,-0.8 - 0.1);
\draw (0.15,0.866-0.05) to [out = 90, in = 0] (0, 1.2) to [out = 180, in = 90] (-0.15,0.866-0.05);
\draw (0.8 + 0.08,-0.8 + 0.15) to [out = -45, in = -20] (0, -1.5) to [out = 160, in = 225] (-0.8 + 0.095,-0.8 - 0.1);
\draw (0.8 - 0.095,-0.8 - 0.1) to [out = -45, in = 20] (0, -1.5) to [out = 200, in = 225]  (-0.8 - 0.08,-0.8 + 0.15);
\end{scope}

\begin{scope}[xshift = 4.5 cm, yshift = 0 cm]
\draw[dashed] (0,-0.18) circle (1 cm);
\draw (-0.15,0.866-0.05) to [out = -90, in = 120] (0.2, 0) to [out = -60, in = 135] (0.8 - 0.095,-0.8 - 0.1);
\draw (0.15,0.866-0.05) to [out = -90, in = 45] (-0.8 - 0.08,-0.8 + 0.15);
\draw (-0.8 + 0.095,-0.8 - 0.1) to [out = 45, in = 135] (0.8 + 0.08,-0.8 + 0.15);
\draw (0.15,0.866-0.05) to [out = 90, in = 0] (0, 1.2) to [out = 180, in = 90] (-0.15,0.866-0.05);
\draw (0.8 + 0.08,-0.8 + 0.15) to [out = -45, in = -20] (0, -1.5) to [out = 160, in = 225] (-0.8 + 0.095,-0.8 - 0.1);
\draw (0.8 - 0.095,-0.8 - 0.1) to [out = -45, in = 20] (0, -1.5) to [out = 200, in = 225]  (-0.8 - 0.08,-0.8 + 0.15);
\end{scope}

\begin{scope}[xshift = 4.5 cm, yshift = -3 cm]
\draw[dashed] (0,-0.18) circle (1 cm);
\draw (0.15,0.866-0.05) to [out = -90, in = 135] (0.8 - 0.095,-0.8 - 0.1);
\draw (-0.15,0.866-0.05) to [out = -90, in = 45] (-0.8 + 0.095,-0.8 - 0.1);
\draw (-0.8 - 0.08,-0.8 + 0.15) to [out = 45, in = 180] (0, -0.5) to [out = 0, in = 135] (0.8 + 0.08,-0.8 + 0.15);
\draw (0.15,0.866-0.05) to [out = 90, in = 0] (0, 1.2) to [out = 180, in = 90] (-0.15,0.866-0.05);
\draw (0.8 + 0.08,-0.8 + 0.15) to [out = -45, in = -20] (0, -1.5) to [out = 160, in = 225] (-0.8 + 0.095,-0.8 - 0.1);
\draw (0.8 - 0.095,-0.8 - 0.1) to [out = -45, in = 20] (0, -1.5) to [out = 200, in = 225]  (-0.8 - 0.08,-0.8 + 0.15);
\end{scope}

\end{scope}

\begin{scope}[xshift = 13.5 cm]
\begin{scope}
\draw[dashed] (0,-0.18) circle (1 cm);
\draw (-0.15,0.866-0.05) to [out = -90, in = 120] (0.2, 0) to [out = -60, in = 135] (0.8 - 0.095,-0.8 - 0.1);
\draw (0.15,0.866-0.05) to [out = -90,  in = 60] (-0.2, 0) to [out = 240, in = 45] (-0.8 + 0.095,-0.8 - 0.1);
\draw (-0.8 - 0.08,-0.8 + 0.15) to [out = 45, in = 180] (0, -0.5) to [out = 0, in = 135] (0.8 + 0.08,-0.8 + 0.15);
\draw (0.15,0.866-0.05) to [out = 90, in = 0] (0, 1.2) to [out = 180, in = 90] (-0.15,0.866-0.05);
\draw (0.8 + 0.08,-0.8 + 0.15) to [out = -45, in = -20] (0, -1.5) to [out = 160, in = 225] (-0.8 + 0.095,-0.8 - 0.1);
\draw (0.8 - 0.095,-0.8 - 0.1) to [out = -45, in = 20] (0, -1.5) to [out = 200, in = 225]  (-0.8 - 0.08,-0.8 + 0.15);
\end{scope}

\end{scope}

\end{tikzpicture}$$
\caption{For any $G'$ corresponding to the third configuration in Figure~\ref{fig:3cyclecases}, $\langle G' \rangle_2(1)=2 - 2 \cdot 2^2 - 2 + 2 \cdot 2^2 +  2 - 2 = 0$.}
\label{fig:3cycleC3}
\end{figure}
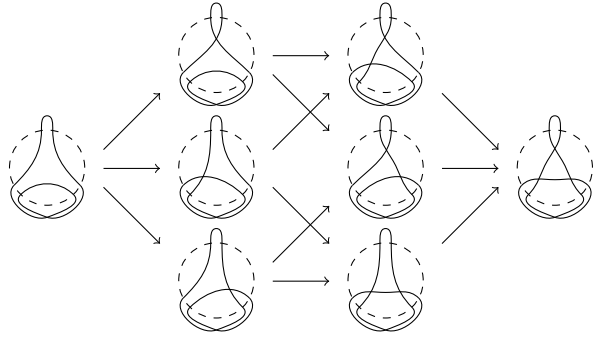

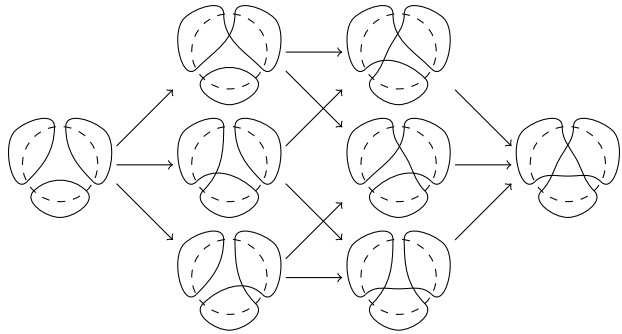
\begin{figure}[H]
$$\begin{tikzpicture}[scale = 0.5]

\begin{scope}
\draw[dashed] (0,-0.18) circle (1 cm);
\draw (0.15,0.866-0.05) to [out = -90, in = 135] (0.8 + 0.08,-0.8 + 0.15);
\draw (-0.15,0.866-0.05) to [out = -90, in = 45] (-0.8 - 0.08,-0.8 + 0.15);
\draw (-0.8 + 0.095,-0.8 - 0.1) to [out = 45, in = 135] (0.8 - 0.095,-0.8 - 0.1);

\draw (0.15,0.866-0.05) to [out = 90, in = 135] (1.2,0.7) to [out = -45, in = -45] (0.8 + 0.08,-0.8 + 0.15);
\draw (-0.15,0.866-0.05) to [out = 90, in = 45] (-1.20, 0.7) to [out = 225, in = 225] (-0.8 - 0.08,-0.8 + 0.15);
\draw (0.8 - 0.095,-0.8 - 0.1) to [out = -45, in = 0] (0, -1.6) to [out = 180, in = 225]  (-0.8 + 0.095,-0.8 - 0.1);
\end{scope}

\begin{scope}[xshift = 1 cm, yshift = -0.2 cm]
\draw[->] (0.5,0.5) --(2,2);
\draw[->] (0.5,0) --(2,0);
\draw[->] (0.5,-0.5) --(2,-2);
\end{scope}

\begin{scope}[xshift = 5.5 cm, yshift = -0.2 cm]
\draw[->] (0.5,0.5) --(2,2);
\draw[->] (0.5,-0.5) --(2,-2);
\draw[->] (0.5,3) --(2,3);
\draw[->] (0.5,-3) --(2,-3);
\draw[->] (0.5,-2.5) --(2,-1);
\draw[->] (0.5,2.5) --(2,1);
\end{scope}

\begin{scope}[xshift = 8.5 cm, yshift = -0.2 cm]
\draw[->] (2,2) -- (3.5,0.5);
\draw[->] (2,0) -- (3.5,0);
\draw[->] (2,-2) -- (3.5,-0.5);
\end{scope}

\begin{scope}[xshift = 4.5 cm, yshift = 3 cm]
\draw[dashed] (0,-0.18) circle (1 cm);
\draw (-0.15,0.866-0.05) to [out = -90, in = 135] (0.8 + 0.08,-0.8 + 0.15);
\draw (0.15,0.866-0.05) to [out = -90, in = 45] (-0.8 - 0.08,-0.8 + 0.15);
\draw (-0.8 + 0.095,-0.8 - 0.1) to [out = 45, in = 135] (0.8 - 0.095,-0.8 - 0.1);

\draw (0.15,0.866-0.05) to [out = 90, in = 135] (1.2,0.7) to [out = -45, in = -45] (0.8 + 0.08,-0.8 + 0.15);
\draw (-0.15,0.866-0.05) to [out = 90, in = 45] (-1.20, 0.7) to [out = 225, in = 225] (-0.8 - 0.08,-0.8 + 0.15);
\draw (0.8 - 0.095,-0.8 - 0.1) to [out = -45, in = 0] (0, -1.6) to [out = 180, in = 225]  (-0.8 + 0.095,-0.8 - 0.1);
\end{scope}

\begin{scope}[xshift = 4.5 cm, yshift = 0 cm]
\draw[dashed] (0,-0.18) circle (1 cm);
\draw (0.15,0.866-0.05) to [out = -90, in = 135] (0.8 + 0.08,-0.8 + 0.15);
\draw (-0.15,0.866-0.05) to [out = -90, in = 45] (-0.8 + 0.095,-0.8 - 0.1);
\draw (-0.8 - 0.08,-0.8 + 0.15) to [out = 45, in = 135] (0.8 - 0.095,-0.8 - 0.1);
\draw (0.15,0.866-0.05) to [out = 90, in = 135] (1.2,0.7) to [out = -45, in = -45] (0.8 + 0.08,-0.8 + 0.15);
\draw (-0.15,0.866-0.05) to [out = 90, in = 45] (-1.20, 0.7) to [out = 225, in = 225] (-0.8 - 0.08,-0.8 + 0.15);
\draw (0.8 - 0.095,-0.8 - 0.1) to [out = -45, in = 0] (0, -1.6) to [out = 180, in = 225]  (-0.8 + 0.095,-0.8 - 0.1);
\end{scope}

\begin{scope}[xshift = 4.5 cm, yshift = -3 cm]
\draw[dashed] (0,-0.18) circle (1 cm);
\draw (0.15,0.866-0.05) to [out = -90, in = 135] (0.8 - 0.095,-0.8 - 0.1);
\draw (-0.15,0.866-0.05) to [out = -90, in = 45] (-0.8 - 0.08,-0.8 + 0.15);
\draw (-0.8 + 0.095,-0.8 - 0.1) to [out = 45, in = 135] (0.8 + 0.08,-0.8 + 0.15);
\draw (0.15,0.866-0.05) to [out = 90, in = 135] (1.2,0.7) to [out = -45, in = -45] (0.8 + 0.08,-0.8 + 0.15);
\draw (-0.15,0.866-0.05) to [out = 90, in = 45] (-1.20, 0.7) to [out = 225, in = 225] (-0.8 - 0.08,-0.8 + 0.15);
\draw (0.8 - 0.095,-0.8 - 0.1) to [out = -45, in = 0] (0, -1.6) to [out = 180, in = 225]  (-0.8 + 0.095,-0.8 - 0.1);
\end{scope}

\begin{scope}[xshift = 4.5 cm]
\begin{scope}[xshift = 4.5 cm, yshift = 3 cm]
\draw[dashed] (0,-0.18) circle (1 cm);
\draw (-0.15,0.866-0.05) to [out = -90, in = 135] (0.8 + 0.08,-0.8 + 0.15);
\draw (0.15,0.866-0.05) to [out = -90,  in = 60] (-0.2, 0) to [out = 240, in = 45] (-0.8 + 0.095,-0.8 - 0.1);
\draw (-0.8 - 0.08,-0.8 + 0.15) to [out = 45, in = 135] (0.8 - 0.095,-0.8 - 0.1);
\draw (0.15,0.866-0.05) to [out = 90, in = 135] (1.2,0.7) to [out = -45, in = -45] (0.8 + 0.08,-0.8 + 0.15);
\draw (-0.15,0.866-0.05) to [out = 90, in = 45] (-1.20, 0.7) to [out = 225, in = 225] (-0.8 - 0.08,-0.8 + 0.15);
\draw (0.8 - 0.095,-0.8 - 0.1) to [out = -45, in = 0] (0, -1.6) to [out = 180, in = 225]  (-0.8 + 0.095,-0.8 - 0.1);
\end{scope}

\begin{scope}[xshift = 4.5 cm, yshift = 0 cm]
\draw[dashed] (0,-0.18) circle (1 cm);
\draw (-0.15,0.866-0.05) to [out = -90, in = 120] (0.2, 0) to [out = -60, in = 135] (0.8 - 0.095,-0.8 - 0.1);
\draw (0.15,0.866-0.05) to [out = -90, in = 45] (-0.8 - 0.08,-0.8 + 0.15);
\draw (-0.8 + 0.095,-0.8 - 0.1) to [out = 45, in = 135] (0.8 + 0.08,-0.8 + 0.15);
\draw (0.15,0.866-0.05) to [out = 90, in = 135] (1.2,0.7) to [out = -45, in = -45] (0.8 + 0.08,-0.8 + 0.15);
\draw (-0.15,0.866-0.05) to [out = 90, in = 45] (-1.20, 0.7) to [out = 225, in = 225] (-0.8 - 0.08,-0.8 + 0.15);
\draw (0.8 - 0.095,-0.8 - 0.1) to [out = -45, in = 0] (0, -1.6) to [out = 180, in = 225]  (-0.8 + 0.095,-0.8 - 0.1);
\end{scope}

\begin{scope}[xshift = 4.5 cm, yshift = -3 cm]
\draw[dashed] (0,-0.18) circle (1 cm);
\draw (0.15,0.866-0.05) to [out = -90, in = 135] (0.8 - 0.095,-0.8 - 0.1);
\draw (-0.15,0.866-0.05) to [out = -90, in = 45] (-0.8 + 0.095,-0.8 - 0.1);
\draw (-0.8 - 0.08,-0.8 + 0.15) to [out = 45, in = 180] (0, -0.5) to [out = 0, in = 135] (0.8 + 0.08,-0.8 + 0.15);
\draw (0.15,0.866-0.05) to [out = 90, in = 135] (1.2,0.7) to [out = -45, in = -45] (0.8 + 0.08,-0.8 + 0.15);
\draw (-0.15,0.866-0.05) to [out = 90, in = 45] (-1.20, 0.7) to [out = 225, in = 225] (-0.8 - 0.08,-0.8 + 0.15);
\draw (0.8 - 0.095,-0.8 - 0.1) to [out = -45, in = 0] (0, -1.6) to [out = 180, in = 225]  (-0.8 + 0.095,-0.8 - 0.1);
\end{scope}

\end{scope}

\begin{scope}[xshift = 13.5 cm]
\begin{scope}
\draw[dashed] (0,-0.18) circle (1 cm);
\draw (-0.15,0.866-0.05) to [out = -90, in = 120] (0.2, 0) to [out = -60, in = 135] (0.8 - 0.095,-0.8 - 0.1);
\draw (0.15,0.866-0.05) to [out = -90,  in = 60] (-0.2, 0) to [out = 240, in = 45] (-0.8 + 0.095,-0.8 - 0.1);
\draw (-0.8 - 0.08,-0.8 + 0.15) to [out = 45, in = 180] (0, -0.5) to [out = 0, in = 135] (0.8 + 0.08,-0.8 + 0.15);
\draw (0.15,0.866-0.05) to [out = 90, in = 135] (1.2,0.7) to [out = -45, in = -45] (0.8 + 0.08,-0.8 + 0.15);
\draw (-0.15,0.866-0.05) to [out = 90, in = 45] (-1.20, 0.7) to [out = 225, in = 225] (-0.8 - 0.08,-0.8 + 0.15);
\draw (0.8 - 0.095,-0.8 - 0.1) to [out = -45, in = 0] (0, -1.6) to [out = 180, in = 225]  (-0.8 + 0.095,-0.8 - 0.1);
\end{scope}

\end{scope}

\end{tikzpicture}$$
\caption{For any $G'$ corresponding to the fourth configuration in Figure~\ref{fig:3cyclecases}, $\langle G' \rangle_2(1)=2^3 -  3 \cdot 2^2 + 3 \cdot 2 - 2 = 0$.}
\label{fig:3cycleC4}
\end{figure}

\begin{figure}[H]
$$\begin{tikzpicture}[scale = 0.5]

\begin{scope}
\draw[dashed] (0,-0.18) circle (1 cm);
\draw (0.15,0.866-0.05) to [out = -90, in = 135] (0.8 + 0.08,-0.8 + 0.15);
\draw (-0.15,0.866-0.05) to [out = -90, in = 45] (-0.8 - 0.08,-0.8 + 0.15);
\draw (-0.8 + 0.095,-0.8 - 0.1) to [out = 45, in = 135] (0.8 - 0.095,-0.8 - 0.1);

\draw (0.15,0.866-0.05) to [out = 90, in = 135] (1.2,0.7) to [out = -45, in = -45] (0.8 - 0.095,-0.8 - 0.1);
\draw (-0.15,0.866-0.05) to [out = 90, in = 45] (-1.20, 0.7) to [out = 225, in = 225] (-0.8 - 0.08,-0.8 + 0.15);
\draw (0.8 + 0.08,-0.8 + 0.15) to [out = -45, in = 0] (0, -1.6) to [out = 180, in = 225]  (-0.8 + 0.095,-0.8 - 0.1);
\end{scope}

\begin{scope}[xshift = 1 cm, yshift = -0.2 cm]
\draw[->] (0.5,0.5) --(2,2);
\draw[->] (0.5,0) --(2,0);
\draw[->] (0.5,-0.5) --(2,-2);
\end{scope}

\begin{scope}[xshift = 5.5 cm, yshift = -0.2 cm]
\draw[->] (0.5,0.5) --(2,2);
\draw[->] (0.5,-0.5) --(2,-2);
\draw[->] (0.5,3) --(2,3);
\draw[->] (0.5,-3) --(2,-3);
\draw[->] (0.5,-2.5) --(2,-1);
\draw[->] (0.5,2.5) --(2,1);
\end{scope}

\begin{scope}[xshift = 8.5 cm, yshift = -0.2 cm]
\draw[->] (2,2) -- (3.5,0.5);
\draw[->] (2,0) -- (3.5,0);
\draw[->] (2,-2) -- (3.5,-0.5);
\end{scope}

\begin{scope}[xshift = 4.5 cm, yshift = 3 cm]
\draw[dashed] (0,-0.18) circle (1 cm);
\draw (-0.15,0.866-0.05) to [out = -90, in = 135] (0.8 + 0.08,-0.8 + 0.15);
\draw (0.15,0.866-0.05) to [out = -90, in = 45] (-0.8 - 0.08,-0.8 + 0.15);
\draw (-0.8 + 0.095,-0.8 - 0.1) to [out = 45, in = 135] (0.8 - 0.095,-0.8 - 0.1);

\draw (0.15,0.866-0.05) to [out = 90, in = 135] (1.2,0.7) to [out = -45, in = -45] (0.8 - 0.095,-0.8 - 0.1);
\draw (-0.15,0.866-0.05) to [out = 90, in = 45] (-1.20, 0.7) to [out = 225, in = 225] (-0.8 - 0.08,-0.8 + 0.15);
\draw (0.8 + 0.08,-0.8 + 0.15) to [out = -45, in = 0] (0, -1.6) to [out = 180, in = 225]  (-0.8 + 0.095,-0.8 - 0.1);
\end{scope}

\begin{scope}[xshift = 4.5 cm, yshift = 0 cm]
\draw[dashed] (0,-0.18) circle (1 cm);
\draw (0.15,0.866-0.05) to [out = -90, in = 135] (0.8 + 0.08,-0.8 + 0.15);
\draw (-0.15,0.866-0.05) to [out = -90, in = 45] (-0.8 + 0.095,-0.8 - 0.1);
\draw (-0.8 - 0.08,-0.8 + 0.15) to [out = 45, in = 135] (0.8 - 0.095,-0.8 - 0.1);

\draw (0.15,0.866-0.05) to [out = 90, in = 135] (1.2,0.7) to [out = -45, in = -45] (0.8 - 0.095,-0.8 - 0.1);
\draw (-0.15,0.866-0.05) to [out = 90, in = 45] (-1.20, 0.7) to [out = 225, in = 225] (-0.8 - 0.08,-0.8 + 0.15);
\draw (0.8 + 0.08,-0.8 + 0.15) to [out = -45, in = 0] (0, -1.6) to [out = 180, in = 225]  (-0.8 + 0.095,-0.8 - 0.1);
\end{scope}

\begin{scope}[xshift = 4.5 cm, yshift = -3 cm]
\draw[dashed] (0,-0.18) circle (1 cm);
\draw (0.15,0.866-0.05) to [out = -90, in = 135] (0.8 - 0.095,-0.8 - 0.1);
\draw (-0.15,0.866-0.05) to [out = -90, in = 45] (-0.8 - 0.08,-0.8 + 0.15);
\draw (-0.8 + 0.095,-0.8 - 0.1) to [out = 45, in = 135] (0.8 + 0.08,-0.8 + 0.15);
\draw (0.15,0.866-0.05) to [out = 90, in = 135] (1.2,0.7) to [out = -45, in = -45] (0.8 - 0.095,-0.8 - 0.1);
\draw (-0.15,0.866-0.05) to [out = 90, in = 45] (-1.20, 0.7) to [out = 225, in = 225] (-0.8 - 0.08,-0.8 + 0.15);
\draw (0.8 + 0.08,-0.8 + 0.15) to [out = -45, in = 0] (0, -1.6) to [out = 180, in = 225]  (-0.8 + 0.095,-0.8 - 0.1);
\end{scope}

\begin{scope}[xshift = 4.5 cm]
\begin{scope}[xshift = 4.5 cm, yshift = 3 cm]
\draw[dashed] (0,-0.18) circle (1 cm);
\draw (-0.15,0.866-0.05) to [out = -90, in = 135] (0.8 + 0.08,-0.8 + 0.15);
\draw (0.15,0.866-0.05) to [out = -90,  in = 60] (-0.2, 0) to [out = 240, in = 45] (-0.8 + 0.095,-0.8 - 0.1);
\draw (-0.8 - 0.08,-0.8 + 0.15) to [out = 45, in = 135] (0.8 - 0.095,-0.8 - 0.1);
\draw (0.15,0.866-0.05) to [out = 90, in = 135] (1.2,0.7) to [out = -45, in = -45] (0.8 - 0.095,-0.8 - 0.1);
\draw (-0.15,0.866-0.05) to [out = 90, in = 45] (-1.20, 0.7) to [out = 225, in = 225] (-0.8 - 0.08,-0.8 + 0.15);
\draw (0.8 + 0.08,-0.8 + 0.15) to [out = -45, in = 0] (0, -1.6) to [out = 180, in = 225]  (-0.8 + 0.095,-0.8 - 0.1);
\end{scope}

\begin{scope}[xshift = 4.5 cm, yshift = 0 cm]
\draw[dashed] (0,-0.18) circle (1 cm);
\draw (-0.15,0.866-0.05) to [out = -90, in = 120] (0.2, 0) to [out = -60, in = 135] (0.8 - 0.095,-0.8 - 0.1);
\draw (0.15,0.866-0.05) to [out = -90, in = 45] (-0.8 - 0.08,-0.8 + 0.15);
\draw (-0.8 + 0.095,-0.8 - 0.1) to [out = 45, in = 135] (0.8 + 0.08,-0.8 + 0.15);
\draw (0.15,0.866-0.05) to [out = 90, in = 135] (1.2,0.7) to [out = -45, in = -45] (0.8 - 0.095,-0.8 - 0.1);
\draw (-0.15,0.866-0.05) to [out = 90, in = 45] (-1.20, 0.7) to [out = 225, in = 225] (-0.8 - 0.08,-0.8 + 0.15);
\draw (0.8 + 0.08,-0.8 + 0.15) to [out = -45, in = 0] (0, -1.6) to [out = 180, in = 225]  (-0.8 + 0.095,-0.8 - 0.1);
\end{scope}

\begin{scope}[xshift = 4.5 cm, yshift = -3 cm]
\draw[dashed] (0,-0.18) circle (1 cm);
\draw (0.15,0.866-0.05) to [out = -90, in = 135] (0.8 - 0.095,-0.8 - 0.1);
\draw (-0.15,0.866-0.05) to [out = -90, in = 45] (-0.8 + 0.095,-0.8 - 0.1);
\draw (-0.8 - 0.08,-0.8 + 0.15) to [out = 45, in = 180] (0, -0.5) to [out = 0, in = 135] (0.8 + 0.08,-0.8 + 0.15);
\draw (0.15,0.866-0.05) to [out = 90, in = 135] (1.2,0.7) to [out = -45, in = -45] (0.8 - 0.095,-0.8 - 0.1);
\draw (-0.15,0.866-0.05) to [out = 90, in = 45] (-1.20, 0.7) to [out = 225, in = 225] (-0.8 - 0.08,-0.8 + 0.15);
\draw (0.8 + 0.08,-0.8 + 0.15) to [out = -45, in = 0] (0, -1.6) to [out = 180, in = 225]  (-0.8 + 0.095,-0.8 - 0.1);
\end{scope}

\end{scope}

\begin{scope}[xshift = 13.5 cm]
\begin{scope}
\draw[dashed] (0,-0.18) circle (1 cm);
\draw (-0.15,0.866-0.05) to [out = -90, in = 120] (0.2, 0) to [out = -60, in = 135] (0.8 - 0.095,-0.8 - 0.1);
\draw (0.15,0.866-0.05) to [out = -90,  in = 60] (-0.2, 0) to [out = 240, in = 45] (-0.8 + 0.095,-0.8 - 0.1);
\draw (-0.8 - 0.08,-0.8 + 0.15) to [out = 45, in = 180] (0, -0.5) to [out = 0, in = 135] (0.8 + 0.08,-0.8 + 0.15);
\draw (0.15,0.866-0.05) to [out = 90, in = 135] (1.2,0.7) to [out = -45, in = -45] (0.8 - 0.095,-0.8 - 0.1);
\draw (-0.15,0.866-0.05) to [out = 90, in = 45] (-1.20, 0.7) to [out = 225, in = 225] (-0.8 - 0.08,-0.8 + 0.15);
\draw (0.8 + 0.08,-0.8 + 0.15) to [out = -45, in = 0] (0, -1.6) to [out = 180, in = 225]  (-0.8 + 0.095,-0.8 - 0.1);
\end{scope}

\end{scope}

\end{tikzpicture}$$
\caption{For any $G'$ corresponding to the fifth configuration in Figure~\ref{fig:3cyclecases}, $\langle G' \rangle_2(1)=2^2 - 2^3 - 2 \cdot 2 + 2 \cdot 2^2 +  2 - 2 = 0$.}
\label{fig:3cycleC5}
\end{figure}
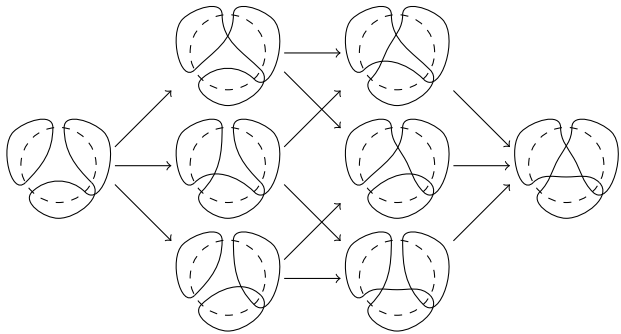

\begin{figure}[H]
$$\begin{tikzpicture}[scale = 0.5]

\begin{scope}
\draw[dashed] (0,-0.18) circle (1 cm);
\draw (0.15,0.866-0.05) to [out = -90, in = 135] (0.8 + 0.08,-0.8 + 0.15);
\draw (-0.15,0.866-0.05) to [out = -90, in = 45] (-0.8 - 0.08,-0.8 + 0.15);
\draw (-0.8 + 0.095,-0.8 - 0.1) to [out = 45, in = 135] (0.8 - 0.095,-0.8 - 0.1);
\draw (0.15,0.866-0.05) to [out = 90, in = 135] (1.2,0.7) to [out = -45, in = -45] (0.8 - 0.095,-0.8 - 0.1);
\draw (-0.15,0.866-0.05) to [out = 90, in = 45] (-1.20, 0.7) to [out = 225, in = 225]  (-0.8 + 0.095,-0.8 - 0.1);
\draw (0.8 + 0.08,-0.8 + 0.15) to [out = -45, in = 0] (0, -1.6) to [out = 180, in = 225] (-0.8 - 0.08,-0.8 + 0.15);
\end{scope}

\begin{scope}[xshift = 1 cm, yshift = -0.2 cm]
\draw[->] (0.5,0.5) --(2,2);
\draw[->] (0.5,0) --(2,0);
\draw[->] (0.5,-0.5) --(2,-2);
\end{scope}

\begin{scope}[xshift = 5.5 cm, yshift = -0.2 cm]
\draw[->] (0.5,0.5) --(2,2);
\draw[->] (0.5,-0.5) --(2,-2);
\draw[->] (0.5,3) --(2,3);
\draw[->] (0.5,-3) --(2,-3);
\draw[->] (0.5,-2.5) --(2,-1);
\draw[->] (0.5,2.5) --(2,1);
\end{scope}

\begin{scope}[xshift = 8.5 cm, yshift = -0.2 cm]
\draw[->] (2,2) -- (3.5,0.5);
\draw[->] (2,0) -- (3.5,0);
\draw[->] (2,-2) -- (3.5,-0.5);
\end{scope}

\begin{scope}[xshift = 4.5 cm, yshift = 3 cm]
\draw[dashed] (0,-0.18) circle (1 cm);
\draw (-0.15,0.866-0.05) to [out = -90, in = 135] (0.8 + 0.08,-0.8 + 0.15);
\draw (0.15,0.866-0.05) to [out = -90, in = 45] (-0.8 - 0.08,-0.8 + 0.15);
\draw (-0.8 + 0.095,-0.8 - 0.1) to [out = 45, in = 135] (0.8 - 0.095,-0.8 - 0.1);
\draw (0.15,0.866-0.05) to [out = 90, in = 135] (1.2,0.7) to [out = -45, in = -45] (0.8 - 0.095,-0.8 - 0.1);
\draw (-0.15,0.866-0.05) to [out = 90, in = 45] (-1.20, 0.7) to [out = 225, in = 225]  (-0.8 + 0.095,-0.8 - 0.1);
\draw (0.8 + 0.08,-0.8 + 0.15) to [out = -45, in = 0] (0, -1.6) to [out = 180, in = 225] (-0.8 - 0.08,-0.8 + 0.15);
\end{scope}

\begin{scope}[xshift = 4.5 cm, yshift = 0 cm]
\draw[dashed] (0,-0.18) circle (1 cm);
\draw (0.15,0.866-0.05) to [out = -90, in = 135] (0.8 + 0.08,-0.8 + 0.15);
\draw (-0.15,0.866-0.05) to [out = -90, in = 45] (-0.8 + 0.095,-0.8 - 0.1);
\draw (-0.8 - 0.08,-0.8 + 0.15) to [out = 45, in = 135] (0.8 - 0.095,-0.8 - 0.1);

\draw (0.15,0.866-0.05) to [out = 90, in = 135] (1.2,0.7) to [out = -45, in = -45] (0.8 - 0.095,-0.8 - 0.1);
\draw (-0.15,0.866-0.05) to [out = 90, in = 45] (-1.20, 0.7) to [out = 225, in = 225]  (-0.8 + 0.095,-0.8 - 0.1);
\draw (0.8 + 0.08,-0.8 + 0.15) to [out = -45, in = 0] (0, -1.6) to [out = 180, in = 225] (-0.8 - 0.08,-0.8 + 0.15);
\end{scope}

\begin{scope}[xshift = 4.5 cm, yshift = -3 cm]
\draw[dashed] (0,-0.18) circle (1 cm);
\draw (0.15,0.866-0.05) to [out = -90, in = 135] (0.8 - 0.095,-0.8 - 0.1);
\draw (-0.15,0.866-0.05) to [out = -90, in = 45] (-0.8 - 0.08,-0.8 + 0.15);
\draw (-0.8 + 0.095,-0.8 - 0.1) to [out = 45, in = 135] (0.8 + 0.08,-0.8 + 0.15);
\draw (0.15,0.866-0.05) to [out = 90, in = 135] (1.2,0.7) to [out = -45, in = -45] (0.8 - 0.095,-0.8 - 0.1);
\draw (-0.15,0.866-0.05) to [out = 90, in = 45] (-1.20, 0.7) to [out = 225, in = 225]  (-0.8 + 0.095,-0.8 - 0.1);
\draw (0.8 + 0.08,-0.8 + 0.15) to [out = -45, in = 0] (0, -1.6) to [out = 180, in = 225] (-0.8 - 0.08,-0.8 + 0.15);
\end{scope}

\begin{scope}[xshift = 4.5 cm]
\begin{scope}[xshift = 4.5 cm, yshift = 3 cm]
\draw[dashed] (0,-0.18) circle (1 cm);
\draw (-0.15,0.866-0.05) to [out = -90, in = 135] (0.8 + 0.08,-0.8 + 0.15);
\draw (0.15,0.866-0.05) to [out = -90,  in = 60] (-0.2, 0) to [out = 240, in = 45] (-0.8 + 0.095,-0.8 - 0.1);
\draw (-0.8 - 0.08,-0.8 + 0.15) to [out = 45, in = 135] (0.8 - 0.095,-0.8 - 0.1);
\draw (0.15,0.866-0.05) to [out = 90, in = 135] (1.2,0.7) to [out = -45, in = -45] (0.8 - 0.095,-0.8 - 0.1);
\draw (-0.15,0.866-0.05) to [out = 90, in = 45] (-1.20, 0.7) to [out = 225, in = 225]  (-0.8 + 0.095,-0.8 - 0.1);
\draw (0.8 + 0.08,-0.8 + 0.15) to [out = -45, in = 0] (0, -1.6) to [out = 180, in = 225] (-0.8 - 0.08,-0.8 + 0.15);
\end{scope}

\begin{scope}[xshift = 4.5 cm, yshift = 0 cm]
\draw[dashed] (0,-0.18) circle (1 cm);
\draw (-0.15,0.866-0.05) to [out = -90, in = 120] (0.2, 0) to [out = -60, in = 135] (0.8 - 0.095,-0.8 - 0.1);
\draw (0.15,0.866-0.05) to [out = -90, in = 45] (-0.8 - 0.08,-0.8 + 0.15);
\draw (-0.8 + 0.095,-0.8 - 0.1) to [out = 45, in = 135] (0.8 + 0.08,-0.8 + 0.15);
\draw (0.15,0.866-0.05) to [out = 90, in = 135] (1.2,0.7) to [out = -45, in = -45] (0.8 - 0.095,-0.8 - 0.1);
\draw (-0.15,0.866-0.05) to [out = 90, in = 45] (-1.20, 0.7) to [out = 225, in = 225]  (-0.8 + 0.095,-0.8 - 0.1);
\draw (0.8 + 0.08,-0.8 + 0.15) to [out = -45, in = 0] (0, -1.6) to [out = 180, in = 225] (-0.8 - 0.08,-0.8 + 0.15);
\end{scope}

\begin{scope}[xshift = 4.5 cm, yshift = -3 cm]
\draw[dashed] (0,-0.18) circle (1 cm);
\draw (0.15,0.866-0.05) to [out = -90, in = 135] (0.8 - 0.095,-0.8 - 0.1);
\draw (-0.15,0.866-0.05) to [out = -90, in = 45] (-0.8 + 0.095,-0.8 - 0.1);
\draw (-0.8 - 0.08,-0.8 + 0.15) to [out = 45, in = 180] (0, -0.5) to [out = 0, in = 135] (0.8 + 0.08,-0.8 + 0.15);
\draw (0.15,0.866-0.05) to [out = 90, in = 135] (1.2,0.7) to [out = -45, in = -45] (0.8 - 0.095,-0.8 - 0.1);
\draw (-0.15,0.866-0.05) to [out = 90, in = 45] (-1.20, 0.7) to [out = 225, in = 225]  (-0.8 + 0.095,-0.8 - 0.1);
\draw (0.8 + 0.08,-0.8 + 0.15) to [out = -45, in = 0] (0, -1.6) to [out = 180, in = 225] (-0.8 - 0.08,-0.8 + 0.15);
\end{scope}

\end{scope}

\begin{scope}[xshift = 13.5 cm]
\begin{scope}
\draw[dashed] (0,-0.18) circle (1 cm);
\draw (-0.15,0.866-0.05) to [out = -90, in = 120] (0.2, 0) to [out = -60, in = 135] (0.8 - 0.095,-0.8 - 0.1);
\draw (0.15,0.866-0.05) to [out = -90,  in = 60] (-0.2, 0) to [out = 240, in = 45] (-0.8 + 0.095,-0.8 - 0.1);
\draw (-0.8 - 0.08,-0.8 + 0.15) to [out = 45, in = 180] (0, -0.5) to [out = 0, in = 135] (0.8 + 0.08,-0.8 + 0.15);

\draw (0.15,0.866-0.05) to [out = 90, in = 135] (1.2,0.7) to [out = -45, in = -45] (0.8 - 0.095,-0.8 - 0.1);
\draw (-0.15,0.866-0.05) to [out = 90, in = 45] (-1.20, 0.7) to [out = 225, in = 225]  (-0.8 + 0.095,-0.8 - 0.1);
\draw (0.8 + 0.08,-0.8 + 0.15) to [out = -45, in = 0] (0, -1.6) to [out = 180, in = 225] (-0.8 - 0.08,-0.8 + 0.15);
\end{scope}

\end{scope}

\end{tikzpicture}$$
\caption{For any $G'$ corresponding to the sixth configuration in Figure~\ref{fig:3cyclecases}, $\langle G' \rangle_2(1)=2 - 2 - 2 \cdot 2^2 + 2 \cdot 2 +  2^3 - 2^2 = 0$.}
\label{fig:3cycleC6}
\end{figure}
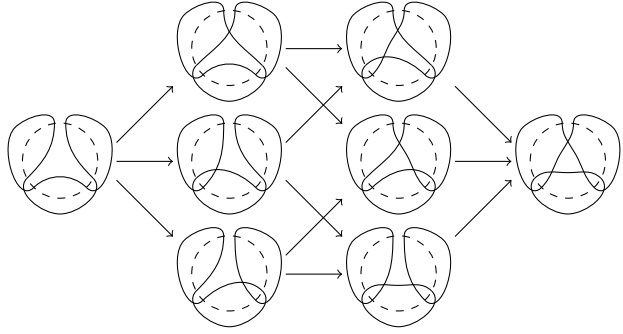

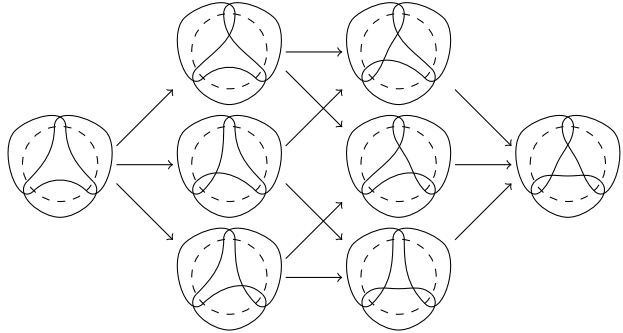
\begin{figure}[H]
$$\begin{tikzpicture}[scale = 0.5]

\begin{scope}
\draw[dashed] (0,-0.18) circle (1 cm);
\draw (0.15,0.866-0.05) to [out = -90, in = 135] (0.8 + 0.08,-0.8 + 0.15);
\draw (-0.15,0.866-0.05) to [out = -90, in = 45] (-0.8 - 0.08,-0.8 + 0.15);
\draw (-0.8 + 0.095,-0.8 - 0.1) to [out = 45, in = 135] (0.8 - 0.095,-0.8 - 0.1);
\draw (-0.15,0.866-0.05) to [out = 90, in = 135] (1.2,0.7) to [out = -45, in = -45] (0.8 - 0.095,-0.8 - 0.1);
\draw (0.15,0.866-0.05) to [out = 90, in = 45] (-1.20, 0.7) to [out = 225, in = 225]  (-0.8 + 0.095,-0.8 - 0.1);
\draw (0.8 + 0.08,-0.8 + 0.15) to [out = -45, in = 0] (0, -1.6) to [out = 180, in = 225] (-0.8 - 0.08,-0.8 + 0.15);
\end{scope}

\begin{scope}[xshift = 1 cm, yshift = -0.2 cm]
\draw[->] (0.5,0.5) --(2,2);
\draw[->] (0.5,0) --(2,0);
\draw[->] (0.5,-0.5) --(2,-2);
\end{scope}

\begin{scope}[xshift = 5.5 cm, yshift = -0.2 cm]
\draw[->] (0.5,0.5) --(2,2);
\draw[->] (0.5,-0.5) --(2,-2);
\draw[->] (0.5,3) --(2,3);
\draw[->] (0.5,-3) --(2,-3);
\draw[->] (0.5,-2.5) --(2,-1);
\draw[->] (0.5,2.5) --(2,1);
\end{scope}

\begin{scope}[xshift = 8.5 cm, yshift = -0.2 cm]
\draw[->] (2,2) -- (3.5,0.5);
\draw[->] (2,0) -- (3.5,0);
\draw[->] (2,-2) -- (3.5,-0.5);
\end{scope}

\begin{scope}[xshift = 4.5 cm, yshift = 3 cm]
\draw[dashed] (0,-0.18) circle (1 cm);
\draw (-0.15,0.866-0.05) to [out = -90, in = 135] (0.8 + 0.08,-0.8 + 0.15);
\draw (0.15,0.866-0.05) to [out = -90, in = 45] (-0.8 - 0.08,-0.8 + 0.15);
\draw (-0.8 + 0.095,-0.8 - 0.1) to [out = 45, in = 135] (0.8 - 0.095,-0.8 - 0.1);
\draw (-0.15,0.866-0.05) to [out = 90, in = 135] (1.2,0.7) to [out = -45, in = -45] (0.8 - 0.095,-0.8 - 0.1);
\draw (0.15,0.866-0.05) to [out = 90, in = 45] (-1.20, 0.7) to [out = 225, in = 225]  (-0.8 + 0.095,-0.8 - 0.1);
\draw (0.8 + 0.08,-0.8 + 0.15) to [out = -45, in = 0] (0, -1.6) to [out = 180, in = 225] (-0.8 - 0.08,-0.8 + 0.15);
\end{scope}

\begin{scope}[xshift = 4.5 cm, yshift = 0 cm]
\draw[dashed] (0,-0.18) circle (1 cm);
\draw (0.15,0.866-0.05) to [out = -90, in = 135] (0.8 + 0.08,-0.8 + 0.15);
\draw (-0.15,0.866-0.05) to [out = -90, in = 45] (-0.8 + 0.095,-0.8 - 0.1);
\draw (-0.8 - 0.08,-0.8 + 0.15) to [out = 45, in = 135] (0.8 - 0.095,-0.8 - 0.1);
\draw (-0.15,0.866-0.05) to [out = 90, in = 135] (1.2,0.7) to [out = -45, in = -45] (0.8 - 0.095,-0.8 - 0.1);
\draw (0.15,0.866-0.05) to [out = 90, in = 45] (-1.20, 0.7) to [out = 225, in = 225]  (-0.8 + 0.095,-0.8 - 0.1);
\draw (0.8 + 0.08,-0.8 + 0.15) to [out = -45, in = 0] (0, -1.6) to [out = 180, in = 225] (-0.8 - 0.08,-0.8 + 0.15);
\end{scope}

\begin{scope}[xshift = 4.5 cm, yshift = -3 cm]
\draw[dashed] (0,-0.18) circle (1 cm);
\draw (0.15,0.866-0.05) to [out = -90, in = 135] (0.8 - 0.095,-0.8 - 0.1);
\draw (-0.15,0.866-0.05) to [out = -90, in = 45] (-0.8 - 0.08,-0.8 + 0.15);
\draw (-0.8 + 0.095,-0.8 - 0.1) to [out = 45, in = 135] (0.8 + 0.08,-0.8 + 0.15);
\draw (-0.15,0.866-0.05) to [out = 90, in = 135] (1.2,0.7) to [out = -45, in = -45] (0.8 - 0.095,-0.8 - 0.1);
\draw (0.15,0.866-0.05) to [out = 90, in = 45] (-1.20, 0.7) to [out = 225, in = 225]  (-0.8 + 0.095,-0.8 - 0.1);
\draw (0.8 + 0.08,-0.8 + 0.15) to [out = -45, in = 0] (0, -1.6) to [out = 180, in = 225] (-0.8 - 0.08,-0.8 + 0.15);
\end{scope}

\begin{scope}[xshift = 4.5 cm]
\begin{scope}[xshift = 4.5 cm, yshift = 3 cm]
\draw[dashed] (0,-0.18) circle (1 cm);
\draw (-0.15,0.866-0.05) to [out = -90, in = 135] (0.8 + 0.08,-0.8 + 0.15);
\draw (0.15,0.866-0.05) to [out = -90,  in = 60] (-0.2, 0) to [out = 240, in = 45] (-0.8 + 0.095,-0.8 - 0.1);
\draw (-0.8 - 0.08,-0.8 + 0.15) to [out = 45, in = 135] (0.8 - 0.095,-0.8 - 0.1);
\draw (-0.15,0.866-0.05) to [out = 90, in = 135] (1.2,0.7) to [out = -45, in = -45] (0.8 - 0.095,-0.8 - 0.1);
\draw (0.15,0.866-0.05) to [out = 90, in = 45] (-1.20, 0.7) to [out = 225, in = 225]  (-0.8 + 0.095,-0.8 - 0.1);
\draw (0.8 + 0.08,-0.8 + 0.15) to [out = -45, in = 0] (0, -1.6) to [out = 180, in = 225] (-0.8 - 0.08,-0.8 + 0.15);
\end{scope}

\begin{scope}[xshift = 4.5 cm, yshift = 0 cm]
\draw[dashed] (0,-0.18) circle (1 cm);
\draw (-0.15,0.866-0.05) to [out = -90, in = 120] (0.2, 0) to [out = -60, in = 135] (0.8 - 0.095,-0.8 - 0.1);
\draw (0.15,0.866-0.05) to [out = -90, in = 45] (-0.8 - 0.08,-0.8 + 0.15);
\draw (-0.8 + 0.095,-0.8 - 0.1) to [out = 45, in = 135] (0.8 + 0.08,-0.8 + 0.15);
\draw (-0.15,0.866-0.05) to [out = 90, in = 135] (1.2,0.7) to [out = -45, in = -45] (0.8 - 0.095,-0.8 - 0.1);
\draw (0.15,0.866-0.05) to [out = 90, in = 45] (-1.20, 0.7) to [out = 225, in = 225]  (-0.8 + 0.095,-0.8 - 0.1);
\draw (0.8 + 0.08,-0.8 + 0.15) to [out = -45, in = 0] (0, -1.6) to [out = 180, in = 225] (-0.8 - 0.08,-0.8 + 0.15);
\end{scope}

\begin{scope}[xshift = 4.5 cm, yshift = -3 cm]
\draw[dashed] (0,-0.18) circle (1 cm);
\draw (0.15,0.866-0.05) to [out = -90, in = 135] (0.8 - 0.095,-0.8 - 0.1);
\draw (-0.15,0.866-0.05) to [out = -90, in = 45] (-0.8 + 0.095,-0.8 - 0.1);
\draw (-0.8 - 0.08,-0.8 + 0.15) to [out = 45, in = 180] (0, -0.5) to [out = 0, in = 135] (0.8 + 0.08,-0.8 + 0.15);
\draw (-0.15,0.866-0.05) to [out = 90, in = 135] (1.2,0.7) to [out = -45, in = -45] (0.8 - 0.095,-0.8 - 0.1);
\draw (0.15,0.866-0.05) to [out = 90, in = 45] (-1.20, 0.7) to [out = 225, in = 225]  (-0.8 + 0.095,-0.8 - 0.1);
\draw (0.8 + 0.08,-0.8 + 0.15) to [out = -45, in = 0] (0, -1.6) to [out = 180, in = 225] (-0.8 - 0.08,-0.8 + 0.15);
\end{scope}

\end{scope}

\begin{scope}[xshift = 13.5 cm]
\begin{scope}
\draw[dashed] (0,-0.18) circle (1 cm);
\draw (-0.15,0.866-0.05) to [out = -90, in = 120] (0.2, 0) to [out = -60, in = 135] (0.8 - 0.095,-0.8 - 0.1);
\draw (0.15,0.866-0.05) to [out = -90,  in = 60] (-0.2, 0) to [out = 240, in = 45] (-0.8 + 0.095,-0.8 - 0.1);
\draw (-0.8 - 0.08,-0.8 + 0.15) to [out = 45, in = 180] (0, -0.5) to [out = 0, in = 135] (0.8 + 0.08,-0.8 + 0.15);
\draw (-0.15,0.866-0.05) to [out = 90, in = 135] (1.2,0.7) to [out = -45, in = -45] (0.8 - 0.095,-0.8 - 0.1);
\draw (0.15,0.866-0.05) to [out = 90, in = 45] (-1.20, 0.7) to [out = 225, in = 225]  (-0.8 + 0.095,-0.8 - 0.1);
\draw (0.8 + 0.08,-0.8 + 0.15) to [out = -45, in = 0] (0, -1.6) to [out = 180, in = 225] (-0.8 - 0.08,-0.8 + 0.15);
\end{scope}

\end{scope}

\end{tikzpicture}$$
\caption{For any $G'$ corresponding to the seventh configuration in Figure~\ref{fig:3cyclecases}, $\langle G' \rangle_2(1)=2 - 3 \cdot 2 + 3 \cdot 2^2 - 2^3 = 0$.}
\label{fig:3cycleC7}
\end{figure}

\end{document}